\documentclass[12pt]{amsart}
 
\pdfoutput=1
\usepackage[margin=1in]{geometry}
\usepackage{amsmath,amsthm,amssymb,amsrefs,bbm,color,esvect,float,graphicx,mathrsfs}
\usepackage{euscript,latexsym}
\usepackage{accents}

\usepackage{epstopdf}
\AppendGraphicsExtensions{.tif}

 \newtheorem{thm}{Theorem}[section]
 \newtheorem{lemma}[thm]{Lemma} \newtheorem{cor}[thm]{Corollary} \newtheorem{rem}[thm]{Remark}
 
 \numberwithin{equation}{section}

\def\dist{{\rm dist}}
\def\supp{{\rm supp \,}}
\def \rdist {{\rm \, rdist}}
\def \inrdist {{\rm \, inrdist}}
\def \ec {{\rm ec}}
\def \child{{\rm ch}}

\def \ch{{\rm ch}}
\def\BMO{{\rm BMO}}
\def\CMO{{\rm CMO}}

\begin{document}

\title[Sparse domination for compactness on weighted spaces]{Sparse domination results for compactness on weighted spaces}
\author{Cody B. Stockdale}
\address{Cody B. Stockdale, Department of Mathematics and Statistics, Washington University in St. Louis, One Brookings Drive, St. Louis, MO, 63130, USA}
\email{codystockdale@wustl.edu}
\author{Paco Villarroya}
\address{Paco Villarroya, Department of Mathematics, Temple University, N. Broad St., Philadelphia, PA, 19122, USA}
\email{paco.villarroya@temple.edu}
\thanks{Villarroya has been partially supported by Spanish Ministerio de Ciencia, Innovaci\'on y Universidades, project grant PGC2018-095366-B-I00}
\author{Brett D. Wick}
\address{Brett D. Wick, Department of Mathematics and Statistics, Washington University in St. Louis, One Brookings Drive, St. Louis, MO, 63130, USA}
\email{bwick@wustl.edu}
\thanks{B. D. Wick's research is supported in part by National Science Foundation DMS grants \#1560955, \#1800057, \#2349868, and \#2054863, and by Australian Research Council grant DP 190100970.}

\begin{abstract}
By means of appropriate sparse bounds, 
we 
deduce compactness on weighted $L^p(w)$ spaces, $1<p<\infty$, for all Calder\'on-Zygmund operators having compact extensions on $L^2(\mathbb{R}^n)$. Similar methods lead to new results on boundedness and compactness of Haar multipliers on weighted spaces. In particular, we prove weighted bounds for weights in a class strictly larger than the typical $A_p$ class.
\end{abstract}

\maketitle


\section{Introduction}
\label{Introduction}

Calder\'on-Zygmund theory is concerned with $L^2(\mathbb{R}^n)$ bounded singular integral operators, $T$, of the form 
$$
    Tf(x)=\int_{\mathbb{R}^n}K(x,y)f(y)\,dy,
$$
where $f$ is 
compactly supported, $x\notin \supp f$, and 
$K$ is a kernel function defined on $(\mathbb{R}^n \times \mathbb{R}^{n}) \setminus \{(x,y):x=y\}$ that, for some $C_K>0$ and $0<\delta \leq 1$, satisfies
$$
    |K(x,y)|\leq \frac{C_K}{|x-y|^n}
$$
whenever $x\neq y$ and
$$
    |K(x,y)-K(x',y')|\leq C_K\frac{|x-x'|^{\delta}+|y-y'|^{\delta}}{|x-y|^{n+\delta}}
$$
whenever $|x-x'|+|y-y'|\leq \frac{1}{2}|x-y|$. The fact that these operators, known as Calder\'on-Zygmund operators, extend to be bounded on $L^p(\mathbb{R}^n)$ for all $1<p<\infty$ is of central importance in harmonic analysis. 

In \cite{HMW1973}, Hunt, Muckenhoupt, and Wheeden extended the  Calder\'on-Zygmund theory to weighted spaces when they characterized the classes of weights, $A_p$, such that the Hilbert transform is bounded on $L^p(w)$ for $1<p<\infty$. A positive almost everywhere and locally integrable function $w$ is an \emph{$A_p$ weight} if $$[w]_{A_p}:=\sup_{Q}\langle w \rangle_Q \langle w^{1-p'}\rangle_Q^{p-1} < \infty,$$
where $\langle w\rangle_Q := \frac{1}{|Q|}\int_Qw(x)\,dx$, $p'$ satisfies $\frac{1}{p}+\frac{1}{p'}=1$, and the supremum is taken over all cubes $Q\subseteq \mathbb{R}^n$ with sides parallel to the coordinate axes.

Shortly later, it was shown that any Calder\'on-Zygmund operator $T$ is bounded on $L^p(w)$ for all $1<p<\infty$ and all $w \in A_p$. However, determining the optimal dependence of $\|T\|_{L^p(w)\rightarrow L^p(w)}$ on $[w]_{A_p}$ was a much more difficult problem. Extrapolation methods allowed for a reduction to the case $p=2$, and the following optimal estimate became known as the $A_2$ conjecture: if $T$ is a Calder\'on-Zygmund operator and $w \in A_2$, then 
$$
    \|Tf\|_{L^2(w)}\lesssim [w]_{A_2}\|f\|_{L^2(w)}
$$
for any $f \in L^2(w)$. 
This question was first solved by Hyt\"onen in the celebrated paper \cite{H2012}. 

In \cite{L2013}, Lerner pursued a different approach to the $A_2$ conjecture using a bound by positive and local operators, called \emph{sparse operators}.
A sparse operator has the form
$$
    Sf:=\sum_{Q\in\mathcal{S}}\langle f\rangle_Q\mathbbm{1}_{Q}
$$ 
for locally integrable $f$, where $\mathcal{S}$ is a collection of cubes satisfying the \emph{sparseness condition}: for every $Q \in \mathcal{S}$, 
$$
    \sum_{P\in \text{ch}_{\mathcal{S}}(Q)} |P|\leq \frac{1}{2}|Q|,
$$ 
where $\text{ch}_{\mathcal{S}}(Q)$ is the set of maximal elements of $\mathcal{S}$ that are strictly contained in $Q$.
A refinement of Lerner's result states that there exists a constant $C>0$ such that for any compactly supported $f \in L^1(\mathbb{R}^n)$, there is a sparse operator $S$ satisfying
$$
    |Tf(x)|\leq C S|f|(x)
$$
for almost every $x\in \supp f$, see \cites{CAR2016,L2017,LN2019}. Since optimal weighted bounds for sparse operators are immediate, this method gives a different proof of the $A_2$ conjecture. Such ``sparse domination'' results have been of immense interest following \cite{L2013}. 

It is natural and of independent interest to study compactness of singular integral operators in addition to the previously described theory concerning boundedness. In \cite{V2015}, the second author began this study by describing necessary and sufficient conditions for Calder\'on-Zygmund operators to extend compactly on $L^p(\mathbb{R})$ for $1<p<\infty $.
Since then, a complete theory for compact Calder\'on-Zygmund operators on 
$L^p(\mathbb{R}^n)$ and the corresponding endpoints has been established, see \cites{PPV2017, OV2017,V2019}. 
As shown in these papers, if a Calder\'on-Zygmund operator extends compactly 
on 
$L^p(\mathbb{R}^n)$, then the kernel $K$ satisfies the estimates 
$$
    |K(x,y)|\lesssim \frac{F_K(x,y)}{|x-y|^n}
$$
whenever $x\neq y$ and
$$
    |K(x,y)-K(x',y')|\leq \frac{|x-x'|^{\delta}+|y-y'|^{\delta}}{|x-y|^{n+\delta}}F_K(x,y)
$$
for some $0<\delta \leq 1$ whenever $|x-x'|+|y-y'|\leq\frac{1}{2}|x-y|$, where $F_K$ is a bounded function
satisfying $$\lim_{|x-y|\rightarrow \infty }F_K(x,y)=\lim_{|x-y|\rightarrow 0 }F_K(x,y)=\lim_{|x+y|\rightarrow \infty }F_K(x,y)=0.$$
The main result of this theory we use here is 
the characterization for compactness of Calder\'on-Zygmund operators at the endpoint case from $L^1(\mathbb{R}^n)$ to $L^{1,\infty} (\mathbb{R}^n)$. The
explicit statement of this result
can be found in 
Theorem \ref{tildeT}
of Section \ref{CZOs}. 

The aim of the current paper is to extend the theory of compact Calder\'on-Zygmund operators on $L^p(\mathbb{R}^n)$ to weighted 
Lebesgue spaces using sparse domination methods.
\begin{thm}
\label{CZOWeightedCompactness}
Let $T$ be a Calder\'on-Zygmund operator that extends compactly on $L^2(\mathbb{R}^n)$. If $1<p<\infty$ and $w \in A_p$, then $T$ extends compactly on $L^p(w)$.
\end{thm}
\noindent The proof of Theorem \ref{CZOWeightedCompactness} involves establishing an appropriate sparse domination result which is interesting in its own right, Theorem \ref{domtilde}. The details are described in Section 3. 

It is worth noting that although our proof of Theorem \ref{CZOWeightedCompactness} is direct, it is possible to achieve weighted compactness results via extrapolation methods, see for example the subsequent paper of Hyt\"onen and Lappas \cite{HL2020}. The sparse technology also allows us to deduce results that are not attainable with extrapolation as we will next describe.   

Motivated by results of \cites{TTV2015,L2017}, we also study properties of Haar multiplier operators. For a bounded sequence of real numbers indexed by the standard dyadic grid of cubes $\mathcal{D}$ on $\mathbb{R}^n$, $\{\varepsilon_Q\}_{Q\in\mathcal{D}}$, the associated \emph{Haar multiplier}, $T$, is given by
$$
    Tf=\sum_{Q\in\mathcal{D}}\varepsilon_Q\langle f,h_Q \rangle h_Q,
$$
where $h_Q$ is the Haar function adapted to $Q$. For generality, we work with Haar multipliers in the setting of arbitrary Radon measures on $\mathbb{R}^n$. See Section 2 for precise details.

Estimates for Haar multiplier operators are often similar to those satisfied by Calder\'on-Zygmund operators but are easier to establish because of a Haar multiplier's diagonal structure. 
In this case, we obtain the following sparse bound. 
\begin{thm}
\label{HaarSparseBound}
Let $T$ be a Haar multiplier adapted to a Radon measure $\mu$ and a bounded sequence of real numbers $\{\varepsilon_Q\}_{Q \in \mathcal{D}}$. Assume that $\mu$ is supported in a dyadic cube $Q_0$. If $f$ is bounded with compact support, then there exists an operator $S_{\varepsilon}$ satisfying 
$$
    |Tf(x)|\lesssim S_{\varepsilon }|f|(x):=\sum_{Q\in \mathcal S}\tilde \varepsilon_Q \langle |f|\rangle_{Q}\mathbbm{1}_{Q}(x)
$$
for almost every $x \in \supp f$,
where $\mathcal{S}$ is a sparse collection of cubes, ${\displaystyle \tilde \varepsilon_{Q}:=\sup_{Q'\in {\mathcal D}(Q)}|\varepsilon_{Q'}|}$, and $\mathcal{D}(Q)$ is the set of dyadic cubes properly contained in $Q$.
\end{thm}

The first consequence of Theorem \ref{HaarSparseBound} is a weighted bound for Haar multipliers with weights in a class strictly larger than $A_p$. For a bounded sequence of real numbers $\{\varepsilon_Q\}_{Q\in \mathcal{D}}$, $0<q<\infty$, and $1<p<\infty$, we say that a nonnegative locally integrable function $w$ is an \emph{$\varepsilon^q A_p$ weight} if 
$$
    [w]_{\varepsilon^q A_p}:=\sup_{Q\in \mathcal{D}} |\varepsilon_{Q}|^q\langle w\rangle_Q \langle w^{1-p'}\rangle_Q^{p-1}<\infty.
$$
Notice that if $\{\varepsilon_Q\}_{Q \in \mathcal{D}}$ is a bounded sequence of real numbers, 
we have 
$
    [w]_{\varepsilon^q A_p}
    \leq {\tilde \varepsilon^q}[w]_{A_p},
$
where $\displaystyle \tilde \varepsilon^q:= \sup_{Q\in\mathcal{D}} |\varepsilon_Q|^q$, and thus $A_p \subseteq \varepsilon^q A_p$. Again, the averages above are taken with respect to a general Radon measure $\mu$. 
\begin{thm}
\label{Haarboundedness}
Let $T$ be a Haar multiplier adapted to a Radon measure $\mu$ and a bounded sequence of real numbers $\{\varepsilon_Q\}_{Q \in \mathcal{D}}$ and let $\displaystyle {\tilde \varepsilon_{Q}:=\sup_{Q'\in {\mathcal D}(Q)}|\varepsilon_{Q'}|}$. 
If $2\leq p<\infty$, and $w \in \tilde \varepsilon A_p$, then $T$ is bounded on $L^p(w)$ with 
$$
    \|Tf\|_{L^p(w)}\lesssim[w]_{\tilde\varepsilon A_p}\|f\|_{L^p(w)}
$$
for all $f \in L^p(w)$; if $1< p\leq 2$ and $w \in \tilde \varepsilon^{p-1}A_p$, then $T$ is bounded on $L^p(w)$ with 
$$
    \|Tf\|_{L^p(w)}\lesssim [w]_{\tilde\varepsilon^{p-1}A_p}^{\frac{p'}{p}}\|f\|_{L^p(w)}
$$
for all $f \in L^p(w)$.
\end{thm}
\noindent We remark that Theorem \ref{Haarboundedness} cannot be obtained by existing extrapolation methods since it holds for weights beyond the $A_p$ classes. 

Moreover, if the coefficients $\varepsilon_Q$ possess extra decay, then we can deduce compactness of the associated Haar multiplier. We use Theorem \ref{HaarSparseBound} to obtain the following. 
\begin{thm}
\label{HaarWeightedCompactness}
Let $T$ be a Haar multiplier adapted to a Radon measure $\mu$ and a bounded sequence of real numbers $\{\varepsilon_Q\}_{Q \in \mathcal{D}}$
such that 
$$\lim_{\ell(Q)\rightarrow \infty}|\varepsilon_{Q}|=\lim_{\ell(Q)\rightarrow 0}|\varepsilon_{Q}|=\lim_{c(Q)\rightarrow \infty }|\varepsilon_{Q}|=0,$$ 
where $\ell(Q)$ and $c(Q)$ denote the side length and center of $Q$ respectively. If $1<p<\infty$ and $w \in A_p$, then $T$ is compact on $L^p(w)$.
\end{thm}

The paper is organized as follows. We prove the sparse bound (Theorem \ref{HaarSparseBound}) and its applications to weighted boundedness (Theorem \ref{Haarboundedness}) and compactness (Theorem \ref{HaarWeightedCompactness}) for Haar multipliers in Section 2. We prove the sparse bound (Theorem \ref{domtilde}) and weighted compactness result (Theorem \ref{CZOWeightedCompactness}) for Calder\'on-Zygmund operators in Section 3.

The authors thank Jos\'e Conde-Alonso for valuable conversations regarding this work. We also thank David Cruz-Uribe for pointing out an issue with Theorem \ref{Haarboundedness} in a previous version of this article, leading to a slightly different result, and thank Timo S. H\"anninen and Emiel Lorist for identifying an issue with Theorem \ref{SparseImpliesWeightedCompactness} in a previous version of this article. We finally thank the anonymous referee for their careful review and feedback.

\section{Haar multipliers}
\label{HaarMultipliers}

\subsection{Definitions and notation}
\label{HaarDef}

Let $\mu$ be a Radon measure on $\mathbb{R}^n$. Throughout this section, all of our integrals, averages, inner products, etcetera will be taken with respect to $\mu$. This will change in Section \ref{CZOs} where we will instead work with Lebesgue measure. 

Let $\mathcal{D}$ denote the standard dyadic grid on $\mathbb{R}^n$, that is, the family of cubes of the form
$Q=\prod_{i=1}^n[2^km_i,2^k(m_i+1))$ for $k, m_i\in \mathbb Z $. 
The expression $\widehat{Q}$ denotes the parent of $Q$, namely, 
the unique dyadic cube such that $\ell(\widehat{Q})=2\ell(Q)$ and 
$Q\subseteq \widehat{Q}$. 
We denote by 
$\text{ch}(Q)$ the children of $Q$, that is, the set of cubes $R\in \mathcal{D}$ such that $\ell(R)=\ell(Q)/2$ and 
$R\subseteq Q$. 
Throughout the paper, all cubes are defined by the tensor product of intervals, and thus their sides are always parallel to the coordinates axes. For $\lambda >0$ and any cube $Q$, we write $\lambda Q$ for the unique cube that satisfies $c(\lambda Q)=c(Q)$
and $\ell(\lambda Q)=\lambda \ell(Q)$. Given a measurable set $\Omega \subseteq \mathbb R^{n}$, we denote by $\mathcal{D}(\Omega)$ the family of dyadic cubes $Q$ such that $Q\subsetneq\Omega$. If $\Omega $ is a dyadic cube, this inclusion is equivalent to $\widehat{Q}\subseteq \Omega$. 

For $Q \in \mathcal{D}$ such that $\mu(Q)>0$, define the \emph{Haar function adapted to $Q$} by
$$
    h_{Q}:=\mu(Q)^{-\frac{1}{2}}\left(\mathbbm{1}_{Q}-\frac{\mu(Q)}{\mu(\widehat{Q})}\mathbbm{1}_{\widehat{Q}}\right).
$$
We note that this notation for $h_Q$ is not standard, but it is convenient for our purposes. Using this notation, $h_{Q}$ is supported on $\widehat{Q}$ and constant on $Q$ and on $\widehat{Q} \setminus Q$.
As shown in \cite{V2019},
we have
$$
f=\sum_{Q\in \mathcal{D}}\langle f, h_Q\rangle h_Q
$$
with convergence in $L^2(\mu )$ norm  
for $f\in L^2(\mu )$ with mean zero, where we write $\langle f,g \rangle :=\int_{\mathbb R^n}fg\,d\mu$. 
\begin{rem}
\label{HaarFrame}
Since 
\begin{equation}\label{disjointnessofHaar}
    \langle h_Q,h_R\rangle
=\delta (\widehat{Q},\widehat{R})\Bigg(\delta (Q,R)
-\frac{\mu(Q)^{\frac{1}{2}}\mu(R)^{\frac{1}{2}}}{\mu(\widehat{Q})}\Bigg),
\end{equation}
where $\delta (Q,R)=1$ if $Q=R$ and zero otherwise, $\{h_Q\}_{Q\in \mathcal{D}}$ is not an orthogonal system. However, $\{h_Q\}_{Q\in\mathcal{D}}$ is a frame for $L^2(\mu)$, namely, there exist $0<C_1\leq C_2$ such that 
$$
C_1\|f\|_{L^2(\mu)}
\leq \Big(\sum_{Q\in {\mathcal D}}\langle f, h_Q\rangle^2\Big)^{\frac{1}{2}}
\leq C_2\|f\|_{L^2(\mu)},
$$
for all $f\in L^2(\mu)$ with mean zero, which is enough to prove our results. The two inequalities above follow directly from \eqref{disjointnessofHaar}.
\end{rem}

Recall that for a bounded sequence of real numbers indexed by dyadic cubes $\{\varepsilon_Q\}_{Q\in \mathcal{D}}$, the associated \emph{Haar multiplier}, $T$, is given by 
$$
    Tf=\sum_{Q\in\mathcal{D}}\varepsilon_Q\langle f,h_Q\rangle h_Q.
$$
The previous equality is understood with almost everywhere pointwise convergence, meaning 
$$
    Tf=\lim_{M\rightarrow \infty} \sum_{Q\in\mathcal{\tilde D}(\mathbb B_{M})}\varepsilon_Q\langle f,h_Q\rangle h_Q,
$$
where $\mathbb B_{M}$ is the ball centered at the origin with diameter $M$ and $\mathcal{\tilde D}(\mathbb B_{M})$ is the finite family of dyadic cubes $Q$ such that both $Q\subsetneq \mathbb B_{M}$ and $\ell(Q)> M^{-1}$.

%

Writing $\langle f \rangle_Q := \frac{1}{\mu(Q)} \int_Q f\,d\mu$, we note that 
\begin{align}\label{product}
\langle f,h_{Q}\rangle h_{Q}&=\left(\langle f\rangle_{Q}-\langle f\rangle_{\widehat{Q}}\right)\left(\mathbbm{1}_{Q}-\frac{\mu(Q)}{\mu(\widehat{Q})}\mathbbm{1}_{\widehat{Q}}\right)
=\langle f\rangle_{Q}\mathbbm{1}_{Q}+a_{Q},
\end{align}
where $a_{Q}:=-\frac{\mu(Q)}{\mu(\widehat{Q})}\langle f\rangle_{Q}\mathbbm{1}_{\widehat{Q}}-\langle f\rangle_{\widehat{Q}}\mathbbm{1}_{Q}+\frac{\mu(Q)}{\mu(\widehat{Q})}\langle f\rangle_{\widehat{Q}}\mathbbm{1}_{\widehat{Q}}$ satisfies the bound 
$$
|a_{Q}|\leq \bigg(\frac{1}{\mu(\widehat{Q})}\int_{Q} |f|\,d\mu+2\langle |f|\rangle_{\widehat{Q}}\bigg)\mathbbm{1}_{\widehat{Q}}\leq 3\langle |f|\rangle_{\widehat{Q}}\mathbbm{1}_{\widehat{Q}}.
$$

\subsection{Technical results}
We use the following \emph{auxiliary maximal function} in the proof of Theorem \ref{HaarSparseBound}.
Let $\{\varepsilon_Q\}_{Q \in \mathcal{D}}$ be a bounded sequence of real numbers indexed by dyadic cubes and define $M_{\varepsilon}$ by $$M_{\varepsilon}f(x):=\sup_{\substack{Q\in {\mathcal D} \\ Q\ni x }}\max_{Q'\in \child(Q)}|\varepsilon_{Q'}| \langle |f|\rangle_Q.
$$
Since trivially
$$
    M_{\varepsilon}f(x)\leq \sup_{\substack{Q\in {\mathcal D}}} 
\max_{Q'\in \child(Q)}|\varepsilon_{Q'}|Mf(x), 
$$
where $M$ is the dyadic Hardy-Littlewood maximal operator defined by $$Mf(x):=\sup_{\substack{Q \in \mathcal{D}\\Q \ni x}}\langle |f|\rangle_Q,$$ we have the following property.
\begin{lemma}
\label{MaxWeakType}
If $\mu$ is a Radon measure supported in a dyadic cube $Q_0$ and $\{\varepsilon_Q\}_{Q \in \mathcal{D}}$ is a bounded sequence of real numbers, then
$$
\| M_{\varepsilon}f\|_{L^{1,\infty}(\mu)}:=
\sup_{\lambda>0}\lambda \, \mu(\{x\in \mathbb R^n : M_{\varepsilon}f(x) > \lambda\})
\lesssim 
\sup_{\substack{Q\in {\mathcal D(Q_0)} 
}}\max_{Q'\in \child(Q)}|\varepsilon_{Q'}|
\| f\|_{L^1(\mu)}
$$
for all $f\in  L^1(\mu)$.
\end{lemma}

We will also use the following \emph{auxiliary maximal truncation Haar multiplier} in the proof of Theorem \ref{HaarSparseBound}. Let $\{\varepsilon_Q\}_{Q \in \mathcal{D}}$ be a bounded sequence of real numbers indexed by dyadic cubes and define $T^{\max}$ by
$$
    T^{\max}f:=\sup_{Q\in\mathcal{D}}\Bigg|\sum_{\substack{P \in \mathcal{D}\\ \widehat{P}\supsetneq Q}} \varepsilon_P\langle f, h_P\rangle h_P\Bigg|.
$$
\begin{lemma}
\label{HaarTruncationWeakType}
If $\mu $ is a Radon measure supported in a dyadic cube $Q_{0}$,  $\{\varepsilon_Q\}_{Q \in \mathcal{D}}$ is a bounded sequence of real numbers, and $T^{\max}$ is defined as above, then
$$
    \|T^{\max}f\|_{L^{1,\infty}(\mu)}\lesssim \sup_{\substack{Q\in {\mathcal D}(Q_0) 
    }}|\varepsilon_{Q}| \|f\|_{L^1(\mu)}
$$ for all $f \in L^1(\mu)$.
\end{lemma}

We will use the following Calder\'on-Zygmund decomposition in the proof of Lemma \ref{HaarTruncationWeakType}. This decomposition is described in \cite{CAP2019}*{Theorem 4.2} and is related to the decomposition of \cite{LSMP2014}*{Theorem 2.1}.

\begin{lemma}
\label{CZDecomposition}
Let $\mu$ be a Radon measure. If $f \in L^1(\mu)$ is nonnegative and $\lambda>0$ (or $\lambda >\frac{\|f\|_{L^1(\mu)}}{\|\mu\|}$ if $\mu$ is a finite measure), then we can write 
$$
    f
    =g+\sum_{j=1}^{\infty}b_j,
$$
where 
\begin{enumerate}
    \item\label{goodfunction} $\|g\|_{L^{2}(\mu)}^2\lesssim \lambda \|f\|_{L^1(\mu)}$, 
    \item there exist pairwise disjoint dyadic cubes $Q_j$ such that $\operatorname{supp}b_j \subseteq \widehat{Q}_j$ and $$\sum_{j=1}^{\infty}\mu(Q_j)\leq \frac{1}{\lambda}\|f\|_{L^1(\mu)},$$ and
    \item $\int_{\mathbb{R}^n}b_j\,d\mu=0$ for each $j$ and $\sum_{j=1}^{\infty}\|b_j\|_{L^1(\mu)} \lesssim \|f\|_{L^1(\mu)}$.
\end{enumerate}
\end{lemma}

\begin{proof}[Proof of Lemma \ref{CZDecomposition}]
We only consider the case when $\mu(\mathbb{R}_j^n)=\infty$ for each $j=1,2,\ldots,2^n$, where the $\mathbb{R}_j^n$ denote the $2^n$ $n$-dimensional quadrants in $\mathbb{R}^n$; the case where at least one quadrant has finite measure is handled using arguments of \cite{LSMP2014}*{Section 3.4}. 
Write 
$$
    \Omega:=\{Mf>\lambda\}=\bigcup_{j=1}^{\infty} Q_j,
$$
where the $Q_j$ are maximal dyadic cubes in the sense that 
$$
    \langle f \rangle_Q \leq \lambda < \langle f \rangle_{Q_j}
$$
whenever $Q\supsetneq Q_j$. Set 
$$
    g:=f\mathbbm{1}_{\mathbb{R}^n\setminus\Omega}+\sum_{j=1}^{\infty} \langle f\mathbbm{1}_{Q_j}\rangle_{\widehat{Q}_j}\mathbbm{1}_{\widehat{Q}_j}
$$
and
$$
    b:=\sum_{j=1}^{\infty}b_j,\quad\text{where}\quad b_j:=f\mathbbm{1}_{Q_j}-\langle f\mathbbm{1}_{Q_j}\rangle_{\widehat{Q}_j}\mathbbm{1}_{\widehat{Q}_j}.
$$
Clearly,
$$
    f=g+b=g+\sum_{j=1}^{\infty}b_j.
$$

To prove (1), write $g=g_1+g_2$ where 
$$
    g_1:=f\mathbbm{1}_{\mathbb{R}^n\setminus\Omega} \quad\text{and}\quad g_2:=\sum_{j=1}^{\infty}\langle f\mathbbm{1}_{Q_j}\rangle_{\widehat{Q}_j}\mathbbm{1}_{\widehat{Q}_j}.
$$
By the Lebesgue Differentiation Theorem, $\|g_1\|_{L^{\infty}(\mu)}\leq \lambda$, and so $\|g_1\|_{L^2(\mu)}^2\leq\lambda\|f\|_{L^1(\mu)}$. On the other hand,
$$
\| g_2\|_{L^2(\mu)}^2
=\sum_{i,j=1}^{\infty}\langle f\mathbbm{1}_{Q_i}\rangle_{\widehat{Q}_i}
\langle f\mathbbm{1}_{Q_j}\rangle_{\widehat{Q}_j}\mu(\widehat{Q}_i\cap \widehat{Q}_j).
$$
Since $\widehat{Q}_i\cap \widehat{Q}_j\in \{
\widehat{Q}_i, \widehat{Q}_j, \emptyset \}$, by symmetry we have
$$
\| g_2\|_{L^2(\mu)}^2
\leq 2\sum_{i=1}^{\infty}\langle f\mathbbm{1}_{Q_i}\rangle_{\widehat{Q}_i}
\sum_{\widehat{Q}_j\subseteq \widehat{Q}_i}
\langle f\mathbbm{1}_{Q_j}\rangle_{\widehat{Q}_j}\mu(\widehat{Q}_j)
= 2\sum_{i=1}^{\infty}\langle f\mathbbm{1}_{Q_i}\rangle_{\widehat{Q}_i}
\sum_{\widehat{Q}_j\subseteq \widehat{Q}_i}
\int_{Q_j} fd\mu.
$$
Now, since $Q_j\subsetneq \widehat{Q}_j\subseteq \widehat{Q}_i$
and since the cubes $Q_j$ are pairwise disjoint 
by maximality, we get
$$
\| g_2\|_{L^2(\mu)}^2
\leq 2 \sum_{i=1}^{\infty}\langle f\mathbbm{1}_{Q_i}\rangle_{\widehat{Q}_i}
\int_{\widehat{Q}_i} fd\mu
= 2 \sum_{i=1}^{\infty}\langle f\rangle_{\widehat{Q}_i}
\int_{Q_i} fd\mu 
\leq 2\lambda \| f\|_{L^{1}(\mu)}.
$$

For property (2), notice that $\operatorname{supp}b_j\subseteq \widehat{Q}_j$ by definition of $b_j$. Also, the cubes $Q_j$ are pairwise disjoint by maximality. With this and the stopping condition $\lambda < \langle f\rangle_{Q_j}$ for each $j$, we have 
$$
    \sum_{j=1}^{\infty} \mu(Q_j) <\sum_{j=1}^{\infty}\frac{1}{\lambda}\|f\mathbbm{1}_{Q_j}\|_{L^1(\mu)}\leq \frac{1}{\lambda}\|f\|_{L^1(\mu)}.
$$

Property (3) follows, first using Fubini's theorem to see
$$
    \int_{\mathbb{R}^n} b_j\,d\mu =\int_{Q_j}f\,d\mu-\int_{\widehat{Q}_j}\langle f\mathbbm{1}_{Q_j}\rangle_{\widehat{Q}_j}\,d\mu =0.
$$
Therefore, 
$$
    \sum_{j=1}^{\infty}\|b_j\|_{L^1(\mu)} \leq \sum_{j=1}^{\infty} \left( \|f\mathbbm{1}_{Q_j}\|_{L^1(\mu)}+\|\langle f\mathbbm{1}_{Q_j}\rangle_{\widehat{Q}_j}\mathbbm{1}_{\widehat{Q}_j}\|_{L^1(\mu)}\right)\lesssim \sum_{j=1}^{\infty}\|f\mathbbm{1}_{Q_j}\|_{L^1(\mu)}\leq \|f\|_{L^1(\mu)}.
$$
\end{proof}

\begin{proof}[Proof of Lemma \ref{HaarTruncationWeakType}]
Since $\supp\mu \subseteq Q_0$, 
we have $\mu(Q)=\mu(\widehat{Q})$ and $\int_{Q}
fd\mu=\int_{\widehat{Q}}fd\mu$ for every dyadic cube $Q$ containing $Q_0$ and every $f\in L^1(\mu)$. With this, we have
\begin{align}\label{product2}
\langle f,h_{Q}\rangle 
&=\mu(Q)^{\frac{1}{2}}(\langle f\rangle_{Q}-\langle f\rangle_{\widehat{Q}})
=0
\end{align}
for such cubes $Q$. By (\ref{product2}) and the fact that $h_Q$ is not defined for dyadic cubes $Q$ such that $\widehat{Q} \cap Q_0 = \emptyset$, we only need to work with cubes
satisfying $\widehat{Q}\subseteq Q_{0}$, or equivalently, cubes in ${\mathcal D}(Q_0)$. 
Furthermore, by the mean zero of $h_P$, we have that 
$$
    T^{\max}\mathbbm{1}_{Q_0}=\sup_{Q\in\mathcal{D}}\Bigg|\sum_{\substack{P \in \mathcal{D}\\Q_0\supseteq  \widehat{P}\supsetneq Q}} \varepsilon_P\langle \mathbbm{1}_{Q_0}, h_P\rangle h_P\Bigg|=0.
$$ 
With this 
$T^{\max}f
=T^{\max}(f-\langle f\rangle _{Q_0}\mathbbm{1}_{Q_0})$
and so, we can assume that $f$ has mean zero.

Let $\displaystyle \varepsilon :=\sup_{\substack{Q\in {\mathcal D(Q_0)} 
}}|\varepsilon_{Q}|$. We wish to show that for all $\lambda>0$ and all $f \in L^1(\mu)$, we have $$\mu(\{T^{\max}f>\lambda\})\lesssim \frac{\varepsilon}{\lambda}\|f\|_{L^1(\mu)}.$$ 

Fix $x\in \mathbb R^{n}$ and $Q \in \mathcal{D}(Q_{0})$. If $x$ is not in the same quadrant of $\mathbb{R}^n$ as $Q$, then $h_P(x)=0$ for every $P$ with $\widehat{P}\supsetneq Q$, and therefore $T^{\max}f(x)=0$. If $x$ and $Q$ are in the same quadrant, let $K$ be the unique dyadic cube containing $x$ such that $\widehat{K}$ is the smallest dyadic cube with $\{x\}\cup Q\subseteq \widehat{K}$. For all $P\in \mathcal{D}$ such that
$Q\subsetneq \widehat{P}\subsetneq \widehat{K}$ we have 
$h_P(x)=0$, and so
\begin{align*}
\Bigg| \sum_{\substack{P\in\mathcal{D}(Q_0)\\ \widehat{P}\supsetneq Q}}\varepsilon_P\langle f,h_P \rangle h_P(x)\Bigg|&=
\Bigg|\sum_{\substack{P\in\mathcal{D}(Q_0) \\ \widehat{P}\supsetneq K}}\varepsilon_P\langle f,h_P \rangle h_P(x)\Bigg|
\\
&=\frac{1}{\mu(K)}\Bigg|\int_{K}\sum_{\substack{P\in\mathcal{D}(Q_0) \\ \widehat{P}\supsetneq K}} \varepsilon_P\langle f,h_P \rangle h_P(y)\,d\mu(y)\Bigg|\\
&=\frac{1}{\mu(K)}\Bigg|\int_{K}\sum_{P\in\mathcal{D}} \varepsilon_P\langle f,h_P \rangle h_P(y)\,d\mu(y)\Bigg|\\
&=|\langle Tf \rangle_{K}|\leq M(Tf)(x),
\end{align*}
where we have used the fact that $\int_K h_P\, d\mu =0$ for $P\in {\mathcal D}$ such that 
$\widehat{P}\cap K=\emptyset $ or $\widehat{P}\subseteq K$.
Taking the supremum over all cubes $Q \in \mathcal{D}$ gives $$T^{\max}f(x) \leq M(Tf)(x).$$

To complete the proof, apply Lemma \ref{CZDecomposition} to $f$ at height $\frac{\lambda}{\varepsilon}$ to write 
$$
    f=g+b=g+\sum_{j=1}^{\infty}b_j,
$$
where properties (1), (2), and (3) of the lemma hold. Moreover, since $f$ has mean zero by assumption and $b$ has mean zero by construction, $g$ also has zero mean.
Then 
$$
    \mu(\{T^{\max}f>\lambda\}) \leq \mu\bigg(\bigg\{T^{\max}g>\frac{\lambda}{2}\bigg\}\bigg)+ \mu\bigg(\bigcup_{j=1}^{\infty}Q_j\bigg)+\mu\bigg(\bigg\{\mathbb{R}^n\setminus \bigcup_{j=1}^{\infty}Q_j: T^{\max}b>\frac{\lambda}{2}\bigg\}\bigg).
$$

Use Chebyshev's inequality, the boundedness of $M$ on $L^2(\mu)$, and property (1) of Lemma \ref{CZDecomposition} to bound the first term as follows:  \begin{align*}
\mu\bigg(\bigg\{T^{\max}g>\frac{\lambda}{2}\bigg\}\bigg)&\lesssim \frac{1}{\lambda^2}\|M(Tg)\|_{L^2(\mu)}^2
\lesssim \frac{1}{\lambda^2}\|Tg\|_{L^2(\mu)}^2
\\
&\lesssim \frac{1}{\lambda^2}\sum_{\substack{Q\in {\mathcal D}(Q_0) 
}}|\varepsilon_{Q}|^{2}|\langle g,h_{Q}\rangle|^{2}
\lesssim \frac{\varepsilon^2}{\lambda^2}\|g\|_{L^2(\mu)}^2
\\
&\lesssim\frac{\varepsilon}{\lambda}\|f\|_{L^1(\mu)}.
\end{align*}


The second term is controlled above by property (2) of Lemma \ref{CZDecomposition}: $$\mu\bigg(\bigcup_{j=1}^{\infty}Q_j\bigg)=\sum_{j=1}^{\infty}\mu(Q_j)\leq \frac{\varepsilon}{\lambda}\|f\|_{L^1(\mu)}.$$

For the third term, we fix $x\in \mathbb R^{n}\backslash \bigcup_{j=1}^{\infty}Q_j$ and $Q \in \mathcal{D}(Q_0)$. By linearity
$$
\Bigg| \sum_{\substack{P \in \mathcal{D}(Q_0) \\ \widehat{P}\supsetneq Q}}\varepsilon_P\langle b,h_P\rangle h_P(x)\Bigg|
\leq  \sum_{j=1}^{\infty}\Bigg| \sum_{\substack{P \in \mathcal{D}(Q_0) \\ \widehat{P}\supsetneq Q}}\varepsilon_P\langle b_j,h_P\rangle h_P(x)\Bigg|.
$$ 
For a fixed index $j$ and fixed $P \in \mathcal{D}(Q_0)$ with $Q\subsetneq \widehat{P}$, we consider three cases:  
\begin{enumerate}
\item[(a)] when $\widehat{Q}_j \subsetneq \widehat{P}$, then $\langle b_j, h_P\rangle=0$ since $h_P$ is constant on $\widehat{Q}_j$ and $b_j$ has mean value zero on $\widehat{Q}_j$, 
\item[(b)] when $\widehat{Q}_j\cap \widehat{P}=\emptyset $, we have $\langle b_j,h_P\rangle=0$ due to their disjoint supports, and
\item[(c)] when $\widehat{P} \subsetneq \widehat{Q}_j$, it must be that $\widehat{P} \subseteq Q_j'$ for some $Q_j' \in \ch{(\widehat{Q}_j)}$. If $Q_j' \neq Q_j$, then $\langle b_j,h_P\rangle =0$ since $b_j$ is constant on $Q_j'$ and $h_P$ has mean value zero on $\widehat{P}\subseteq Q_j'$. If $\widehat{P} \subseteq Q_j$, then $h_P(x)=0$ since $x \not \in Q_j$ and $\text{supp}(h_P) \subseteq \widehat{P}$.
\end{enumerate}
We are left with the case $\widehat{P}=\widehat{Q}_j$, 
so $$
    \Bigg| \sum_{\substack{P \in \mathcal{D}(Q_0) \\ \widehat{P}\supsetneq Q}}\varepsilon_P\langle b,h_P\rangle h_P(x)\Bigg| \leq \sum_{j=1}^{\infty}\Bigg|\sum_{\substack{P \in \mathcal{D}(Q_0)\\ \widehat{P}\supseteq Q \\ \widehat{P}=\widehat{Q}_j}} \varepsilon_P\langle b_j,h_P \rangle h_P(x)\Bigg| \lesssim \varepsilon\sum_{j=1}^{\infty}\sum_{\substack{P\in\mathcal{D}\\ \widehat{P}=\widehat{Q}_j}}|\langle b_j,h_P\rangle||h_P(x)|.
$$

Taking the supremum over $Q \in \mathcal{D}$, we have 
$$
    T^{\max}b(x)\leq \varepsilon\sum_{j=1}^{\infty}\sum_{\substack{P\in\mathcal{D}\\ \widehat{P}=\widehat{Q}_j}}|\langle b_j,h_P\rangle| |h_p(x)|
$$ 
for $x \not \in \bigcup_{j=1}^{\infty}Q_j$. 
Now, by definition
$$
    |\langle b_j , h_P \rangle| = \mu(P)^{-\frac{1}{2}}\Big|\int_{\mathbb{R}^n}f\mathbbm{1}_{Q_j}\mathbbm{1}_P-\langle f\mathbbm{1}_{Q_j}\rangle_{\widehat{Q}_j}\mathbbm{1}_{P}-\frac{\mu(P)}{\mu(\widehat{P})}f\mathbbm{1}_{Q_j}+\frac{\mu(P)}{\mu(\widehat{P})}\langle f\mathbbm{1}_{Q_j}\rangle_{\widehat{Q}_j}\mathbbm{1}_{\widehat{Q}_j}\,d\mu \Big| ,
$$
and since $\widehat{P}=\widehat{Q}_j$, we have 
$$
    |\langle b_j,h_P\rangle| \lesssim \mu(P)^{-\frac{1}{2}}\|f\mathbbm{1}_{Q_j}\|_{L^1(\mu)}.
$$
On the other hand, 
\begin{align*}
\|h_P\|_{L^1(\mu)}
&= \mu(P)^{-\frac{1}{2}}
\int_{\mathbb{R}^n} \bigg(\mathbbm{1}_{P}(x)-\frac{\mu(P)}{\mu(\widehat{P})}\mathbbm{1}_{\widehat{P}}(x)\bigg)\,d\mu(x)
\leq 2\mu(P)^{\frac{1}{2}}
\end{align*}
Therefore using Chebyshev's inequality and the above estimates, we have
\begin{align*}
    \mu\bigg(\bigg\{\mathbb{R}^n\setminus\bigcup_{j=1}^{\infty}Q_j: T^{\max}b>\frac{\lambda}{2}\bigg\}\bigg)&\lesssim \frac{1}{\lambda}\int_{\mathbb{R}^n\setminus \bigcup_{j=1}^{\infty}Q_j} T^{\max}b(x)\,d\mu(x)\\
    &\leq \frac{\varepsilon}{\lambda}\int_{\mathbb{R}^n\setminus \bigcup_{j=1}^{\infty}Q_j} \sum_{j=1}^{\infty}\sum_{\substack{P\in\mathcal{D}\\ \widehat{P}=\widehat{Q}_j}}|\langle b_j,h_P\rangle|
    |h_P(x)|
    \\
    &\lesssim \frac{\varepsilon}{\lambda}\sum_{j=1}^{\infty}\sum_{\substack{P\in\mathcal{D}\\ \widehat{P}=\widehat{Q}_j}}|\langle b_j,h_P\rangle|\mu(P)^{\frac{1}{2}}\\
    &\lesssim \frac{\varepsilon}{\lambda}\sum_{j=1}^{\infty}\sum_{\substack{P\in\mathcal{D}\\ \widehat{P}=\widehat{Q}_j}}\|f\mathbbm{1}_{Q_j}\|_{L^1(\mu)}\\
    &\lesssim \frac{\varepsilon}{\lambda}\|f\|_{L^1(\mu)}.
\end{align*}

Combining all estimates gives
$$
    \mu(\{T^{\max}f>\lambda\})\lesssim\frac{\varepsilon}{\lambda}\|f\|_{L^1(\mu)}.
$$
\end{proof}
\begin{rem}
\label{TfPointwiseControl}
We note for later reference that $|Tf|\leq T^{\max}f$ pointwise. Indeed, by definition and using that $\supp\mu\subseteq Q_0$, we have 
$$
|Tf(x)|=\lim_{\substack{\ell(Q)\rightarrow 0\\x\in Q}}\Bigg|\sum_{\substack{P \in \mathcal{D}(Q_{0})\\ \widehat{P}\supseteq Q}} \varepsilon_P\langle f, h_P\rangle h_P(x)\Bigg|
\leq \sup_{Q\in \mathcal{D}}\Bigg|\sum_{\substack{P \in \mathcal{D}(Q_{0})\\ \widehat{P}\supseteq Q}} \varepsilon_P\langle f, h_P\rangle h_P(x)\Bigg|=T^{\max}f(x).
$$
\end{rem}

\subsection{Sparse domination}
\label{HaarSparse}

We are ready to prove the sparse bound for Haar multipliers. Our proof follows closely the ideas of \cite{L2017}. Differences include using the auxiliary maximal function $M_{\varepsilon}$, using the Haar wavelet frame $\{h_Q\}_{Q\in \mathcal{D}}$, and tracking the role of the coefficients $\varepsilon_Q$ throughout.

\begin{proof}[Proof of Theorem \ref{HaarSparseBound}]
Since $\langle f,h_Q\rangle =0$ for all $Q\in {\mathcal D}$ such that $Q_0\subsetneq \widehat{Q}$ or $Q_0\cap \widehat{Q}=\emptyset $, 
we only need to work with cubes 
in $\mathcal D(Q_{0})$, meaning that 
\begin{equation}\label{TQ}
Tf=\sum_{Q\in {\mathcal D}(Q_0)}
\varepsilon_Q\langle f, h_{Q}\rangle h_{Q}=:T_{Q_0}f.
\end{equation}

We start by adding the cube $Q_{0}$ to the family $\mathcal S$ and the function 
$\tilde \varepsilon_{Q_{0}} \langle |f|\rangle_{Q_{0}}\mathbbm{1}_{Q_{0}}$ to the sparse operator 
$S_{\varepsilon }$.
Define 
$$
E_{Q_{0}}:=\left\{x\in Q_{0}: \max\left\{M_{\varepsilon}f(x),T^{\max}f(x)\right\}> 2C\tilde \varepsilon_{Q_{0}} \langle |f|\rangle_{Q_{0}}\right\}.
$$
where $C>0$ is the sum of the implicit constants in 
Lemma \ref{MaxWeakType} and Lemma \ref{HaarTruncationWeakType}. By these two results, we have
$$
\mu(E_{Q_{0}})\leq \frac{1}{2C \tilde \varepsilon_{Q_{0}}  \langle |f|\rangle_{Q_{0}}} C
\max_{Q\in \mathcal D(Q_{0})} |\varepsilon_{Q}| \| f\|_{L^1(\mu)}
\leq \frac{1}{2}\mu(Q_0).
$$



Let $\mathcal E_{Q_{0}} $ be the family of maximal dyadic cubes $P$ contained in $E_{Q_{0}}$. 
For $x\in Q_{0}\backslash E_{Q_{0}}$ we trivially have 
$$
|Tf(x)|\leq T^{\max}f(x)\leq 2C\tilde \varepsilon_{Q_{0}} \langle |f|\rangle_{Q_{0}}
\mathbbm{1}_{Q_{0}}(x)\lesssim S_{\varepsilon} |f|(x).
$$
Otherwise, consider $x \in E_{Q_0}$. Let  $P\in \mathcal E_{Q_{0}}$ be the unique cube such that $x\in P\subseteq Q_{0}$. We formally decompose $Tf$ as follows:
\begin{align*}
Tf&=\sum_{\substack{I\in \mathcal D\\ \widehat{I}\supsetneq \widehat{P}}}\varepsilon_{I}
\langle f, h_{I}\rangle h_{I}
+\sum_{
\substack{I\in \mathcal D\\ I\in \text{ch}(\widehat{P})}}\varepsilon_{I}\langle f, h_{I}\rangle h_{I}
+\sum_{\substack{I\in \mathcal D\\ \widehat{I}\subsetneq \widehat{P}}}\varepsilon_{I}\langle f, h_{I}\rangle h_{I}
+\sum_{\substack{I\in \mathcal D\\ \widehat{I}\cap \widehat{P}=\emptyset 
}}\varepsilon_{I}\langle f, h_{I}\rangle h_{I}.
\end{align*}
In the third term,  
$\widehat{I}\subsetneq \widehat{P}$ implies $\widehat{I}\subseteq P'$ for some $P'\in \child(\widehat{P})$, and so $I\in \mathcal D(P')$. 
In the last term, 
$\widehat{I}\cap \widehat{P}=\emptyset $
implies $\widehat{I}\cap P=\emptyset$. And since $x \in P$ and $\supp h_I \subseteq \widehat{I}$, we have
$h_{I}(x)=0$ 
for all cubes $I\in \mathcal D$ with $\widehat{I}\cap P=\emptyset$. This implies that the fourth term vanishes. Then
\begin{align*}
Tf(x)=&\sum_{\substack{I\in \mathcal D\\ \widehat{I}\supsetneq \widehat{P}}}\varepsilon_{I}
\langle f, h_{I}\rangle h_{I}(x)
+\sum_{I\in \child(\widehat{P})}\varepsilon_{I}\langle f, h_{I}\rangle h_{I}(x)
+\sum_{P'\in \child(\widehat{P})}\sum_{I\in \mathcal D(P')}\varepsilon_{I}\langle f, h_{I}\rangle h_{I}(x)
\end{align*}
Now we use the decomposition in \eqref{product} and the facts that 
$h_{I}(x)=0$ if $I\in \mathcal D(P')$ for $P'\in \child(\widehat{P})\setminus \{P\}$, and 
$\mathbbm{1}_{I}(x)=0$ if $I\in \child(\widehat{P})\setminus \{P\}$, 
to write
\begin{align}\label{decomp}
\nonumber
Tf(x)&
=\sum_{\substack{I\in \mathcal D\\ \widehat{I}\supsetneq \widehat{P}}}\varepsilon_{I}\langle f, h_{I}\rangle h_{I}(x)
+\sum_{I\in \child(\widehat{P})}(\varepsilon_{I}a_{I}(x)+\varepsilon_{I}\langle f \rangle_{I}\mathbbm{1}_{I}(x))
+\sum_{I\in \mathcal D(P)}\varepsilon_{I}\langle f, h_{I}\rangle h_{I}(x)
\\
&
=\sum_{\substack{I\in \mathcal D\\ \widehat{I}\supsetneq \widehat{P}}}\varepsilon_{I}\langle f, h_{I}\rangle h_{I}(x)
+\sum_{I\in \child(\widehat{P})}\varepsilon_{I}a_{I}(x)+\varepsilon_{P}\langle f \rangle_{P}\mathbbm{1}_{P}(x)
+T_Pf(x),
\end{align}
where $T_P$ is as defined in \eqref{TQ}. 

By maximality of $P$, there exists a point $y\in \widehat{P}\setminus E_{Q_{0}}$. The first term in \eqref{decomp} can be bounded as follows:
$$
\Bigg|\sum_{\substack{I\in \mathcal D\\ \widehat{I}\supsetneq \widehat{P}}}\varepsilon_{I}
\langle f, h_{I}\rangle h_{I}(x)\Bigg|
=\Bigg| \sum_{\substack{I\in \mathcal D\\ \widehat{I}\supsetneq \widehat{P}}}\varepsilon_{I}
\langle f, h_{I}\rangle h_{I}(y)\Bigg|
\leq T^{\max}f(y)\leq 2C\tilde \varepsilon_{Q_0} \langle f\rangle_{Q_{0}}.
$$
Similarly, since $|a_{I}(x)|\leq 3\langle |f| \rangle_{\widehat{P}}\mathbbm{1}_{\widehat{P}}(x)$ for $I\in \child(\widehat{P})$, we have 
for the second term 
\begin{align*}
\Big|\sum_{I\in \child(\widehat{P})}\varepsilon_{I}a_{I}(x)
\Big|
&\lesssim \sum_{I\in \child(\widehat{P})}|\varepsilon_{I}| \langle |f| \rangle_{\widehat{P}}
 \mathbbm{1}_{\widehat{P}}(x)
\leq 2^{n}\max_{I\in \child(\widehat{P})}|\varepsilon_{I} |\langle |f| \rangle_{\widehat{P}}\mathbbm{1}_{\widehat{P}}(x)
\\
&
= 2^{n}\max_{I\in \child(\widehat{P})}|\varepsilon_{I} |\langle |f| \rangle_{\widehat{P}}\mathbbm{1}_{\widehat{P}}(y)
\leq 2^{n}M_{\varepsilon}f(y)\leq 2^{n+1}C\tilde \varepsilon_{Q_{0}} \langle |f|\rangle_{Q_{0}}.
\end{align*}



For the third term, 
we directly add the cubes $P\in \mathcal{E}_{Q_{0}}$ to the family $\mathcal S$ and the functions 
$ \tilde \varepsilon_{P}\langle |f|\rangle_{P}\mathbbm{1}_{P}$ to $S_{\varepsilon}|f|$. By disjointness, the sparseness condition holds for $Q_0$:
$$
\sum_{P\in \mathcal E_{Q_0}} \mu(P)\leq \mu(E_{Q_0})\leq \frac{1}{2}\mu(Q_0).
$$ 
The 
last term in \eqref{decomp} is treated by repeating the previous reasoning applied to $T_Pf$ instead of $Tf=T_{Q_0}f$, that is, 
starting the argument with 
$$
E_{P}:=\left\{x\in P: \max\left\{M_{\varepsilon}f(x),T_P^{\max}f(x)\right\}> 2C\tilde \varepsilon_{P} \langle |f|\rangle_{P}\right\}
$$
and adding to ${\mathcal S}$ 
the family $\mathcal E_{P}$ of maximal dyadic cubes contained in $E_{P}$. 
\end{proof}

\subsection{Boundedness and compactness on weighted spaces}
\label{HaarCompactness}

We now study boundedness and compactness of Haar multiplier operators on weighted spaces. We first prove the boundedness result Theorem \ref{Haarboundedness}.




We will use the following \emph{dyadic maximal function adapted to weights}. Given a locally integrable and positive almost everywhere function $w$, we define $M_{w}$ by 
$$
    M_{w}f(x)
    :=\sup_{\substack{Q\in \mathcal{D}\\ x\in Q}}\frac{1}{w(Q)}\int_Q |f|w\,d\mu,
$$
where $w(Q):=\int_Q w\,d\mu$. The following lemma is well-known, see \cites{LSMP2014,M2012} for example.
\begin{lemma}
\label{WeightedMaxBoundedness}
If $w$ is locally integrable and positive almost everywhere and $1<p<\infty$, then $M_{w}$ is bounded from $L^p(w)$ to itself. Moreover, $\|M_{w }\|_{L^p(w)\rightarrow L^p(w)}$ does not depend on $w$. 
\end{lemma}

\begin{proof}[Proof of Theorem \ref{Haarboundedness}]
Our proof closely follows the argument in \cite{M2012}. 

It is enough to consider the case where $\mu$ is compactly supported, as long as we obtain bounds that are independent of $\supp\mu$. Assuming $\mu $ has compact support, there exist pairwise disjoint dyadic cubes $\{Q^k\}_{k=1}^{2^n}$ with each $Q^k$ in one of the quadrants of $\mathbb{R}^n$ and such that $\supp \mu \subseteq \bigcup_{k=1}^{2^n}Q^k$. Dividing $\mu$ into the $2^n$ measures $\mu_k(A):=\mu(A\cap Q^k)$, we can further assume that $\supp \mu$ is contained in a dyadic cube.

Suppose that $p\ge 2$ and set $\sigma=w^{1-p'}$. We use the equivalence 
$$
    \|T\|_{L^p(w)\rightarrow L^p(w)} = \|T(\cdot\,\sigma)\|_{L^p(\sigma)\rightarrow L^p(w)}
$$
and proceed by duality. Let $f\in L^p(\sigma)$ and $g\in L^{p'}(w)$ be nonnegative functions with compact support. 
Apply 
Theorem \ref{HaarSparseBound} 
to obtain the estimate
$$
    \langle |T(f\sigma)|, gw\rangle 
    \lesssim \langle S_{\varepsilon}(f\sigma),gw\rangle 
    =\sum_{j,k} \tilde \varepsilon_{Q_{j}^{k}}\langle f\sigma\rangle_{Q_j^k}\langle gw\rangle_{Q_j^k}\mu(Q_j^k),
$$
where we denote the cubes in the sparse collection $\mathcal{S}$ chosen at the step $k$ by $Q_j^k$. 
Note that, although the cubes $Q_j^k$ 
and the coefficients $\tilde \varepsilon_{Q_{j}^{k}}$
depend on $\supp \mu$, we aim for final estimates that are independent of $\supp\mu$. 

Define $E_j^k:=Q_j^k\setminus\left(\bigcup_{j}Q_j^{k+1}\right)$. Notice that the sparseness property of the cubes $Q_j ^k$ implies that the sets $E_j^k$ are pairwise disjoint and that $\mu(Q_j^k)\leq 2 \mu(E_j^k)$.  
Using the latter inequality, the $\tilde\varepsilon A_p$ condition for $w$, and the containment $E_j^k\subseteq Q_j^k$, we have
\begin{align}\label{estmax}
\nonumber
    \langle S_{\varepsilon}(f\sigma),gw\rangle
    &= \sum_{j,k} \tilde \varepsilon_{Q_{j}^{k}}\langle f\sigma\rangle_{Q_j^k}\int_{Q_j^k} gw\,d\mu\\
    \nonumber
    &=\sum_{j,k}\tilde\varepsilon_{Q_j^k} \frac{w(Q_j^k)\sigma(Q_j^k)^{p-1}}{\mu(Q_j^k)^p}\frac{\mu(Q_j^k)^{p-1}}{w(Q_j^k)\sigma(Q_j^k)^{p-1}}\int_{Q_j^k}f\sigma\,d\mu \int_{Q_j^k}gw\,d\mu\\
    \nonumber
    &\leq [w]_{\tilde \varepsilon A_p} \sum_{j,k}\left(\frac{1}{\sigma(Q_j^k)}\int_{Q_j^k}f\sigma d\mu\right)\left(\frac{1}{w(Q_j^k)}\int_{Q_j^k}gw\,d\mu\right)\mu(Q_j^k)^{p-1}\sigma(Q_j^k)^{2-p}\\
    \nonumber
    &\leq 2^{p-1}[ w]_{\tilde\varepsilon A_p}\sum_{j,k}\left(\frac{1}{\sigma(Q_j^k)}\int_{Q_j^k}f\sigma d\mu\right)\left(\frac{1}{w(Q_j^k)}\int_{Q_j^k}gw\,d\mu\right)\mu(E_j^k)^{p-1}\sigma(E_j^k)^{2-p}.
\end{align}
By H\"older's inequality, 
$$
    \mu(E_j^k)\leq w(E_j^k)^{\frac{1}{p}}\sigma(E_j^k)^{\frac{1}{p'}},
$$
and so
$$
    \mu(E_j^k)^{p-1}\sigma(E_j^k)^{2-p}\leq
    w(E_j^k)^{\frac{p-1}{p}}\sigma(E_j^k)^{\frac{p-1}{p'}}\sigma(E_j^k)^{2-p}
    =w(E_j^k)^{\frac{1}{p'}}\sigma(E_j^k)^{\frac{1}{p}},
$$
since $\frac{p-1}{p'}+2-p=\frac{1}{p}$.
Using the estimates above, H\"older's inequality, the disjointness of the sets $E_j^k$, and Lemma \ref{WeightedMaxBoundedness}, we bound $\langle S_{\varepsilon}(f\sigma),gw\rangle$ by a constant times
\begin{align}
    \nonumber
    &[w]_{\tilde\varepsilon A_p}\sum_{j,k}\left(\frac{1}{\sigma(Q_j^k)}\int_{Q_j^k}f\sigma d\mu\right)\left(\frac{1}{w(Q_j^k)}\int_{Q_j^k}gw\,d\mu\right)w(E_j^k)^{\frac{1}{p'}}\sigma(E_j^k)^{\frac{1}{p}}\\
    \nonumber
     &\leq[w]_{\tilde\varepsilon A_p}\left(\sum_{j,k}\left(\frac{1}{\sigma(Q_j^k)}\int_{Q_j^k}f\sigma d\mu\right)^p\sigma(E_j^k)\right)^{\frac{1}{p}}\left(\sum_{j,k}\left(\frac{1}{w(Q_j^k)}\int_{Q_j^k}gw\,d\mu\right)^{p'}w(E_j^k)\right)^{\frac{1}{p'}}\\
    \nonumber&\leq [w]_{\tilde\varepsilon A_p}\|M_{\sigma}f\|_{L^p(\sigma)}\|M_{w}g\|_{L^{p'}(w)}\\
    \nonumber
    &\lesssim [w]_{\tilde\varepsilon A_p}\|f\|_{L^p(\sigma)}\|g\|_{L^{p'}(w)}.
\end{align}
The case $1<p<2$ follows from duality since $w \in \tilde\varepsilon^{p-1} A_p$ if and only if $\sigma \in \tilde\varepsilon A_{p'}$, and $[\sigma]_{\tilde\varepsilon A_{p'}}=[w]_{\tilde\varepsilon^{p-1}A_p}^{\frac{p'}{p}}$. Thus
$$
    \|T\|_{L^p(w)\rightarrow L^p(w)} = \|T^*\|_{L^{p'}(\sigma)\rightarrow L^{p'}(\sigma)}\lesssim [\sigma]_{\tilde\varepsilon A_{p'}}=[w]_{\tilde\varepsilon^{p-1} A_p}^{\frac{p'}{p}}.
$$
\end{proof}

In the second half of this subsection, we treat the compactness of Haar multipliers. 
We start with some definitions. 
For any positive integer $N$, let $\mathcal{D}_N$ denote the \emph{lagom cubes} 
$$
    \mathcal{D}_N:= \{Q\in\mathcal{D}: -2^N\leq l(Q)\leq 2^N \,\,\text{and} \,\, \text{rdist}(Q,\mathbb{B}_{2^N})\leq N\},
$$ where $\text{rdist(P,Q)}:=1+\frac{\dist(P,Q)}{\max\{\ell(P),\ell(Q)\}}$ and 
$\mathbb{B}_{2^N}$ is the ball centered at the origin with radius $2^N$. We write $\mathcal D_{N}^{c}:=\mathcal D\backslash \mathcal D_{N}$. 

\begin{thm}
\label{SparseImpliesWeightedCompactness} 
Let $1<p<\infty$ and $w \in A_p$. If $T$ is a linear operator and $\{T_N\}_{N =1}^{\infty}$ is a sequence of compact operators on $L^p(w)$ satisfying for every $\epsilon>0$, there exists $N_0 \in \mathbb{N}$ such that for all $N > N_0$ and all bounded $f$ with compact support, there exists a sparse collection $\mathcal{S}$ with 
$$
	|(T-T_N)f(x)|\leq \epsilon \sum_{Q\in\mathcal{S}}\langle|f|\rangle_Q\mathbbm{1}_Q(x)
$$
for almost every $x \in \supp f$, then $T$ extends compactly on $L^p(w)$.
\end{thm} 

\begin{proof}Since uniform limits of compact operators are compact, it suffices to show that 
$$
	\lim_{N\rightarrow \infty}\|T-T_N\|_{L^p(w)\rightarrow L^p(w)} = 0.
$$
Let $\epsilon>0$ and fix $N_0 \in \mathbb{N}$ according to the hypotheses. Let $f \in L^p(w)$ and assume without loss of generality that $f$ is nonnegative, bounded, and has compact support. Then, for any $N>N_0$, there exists a sparse collection $\mathcal{S}$ such that 
$$
	|(T-T_N)f(x)|\leq \epsilon\sum_{Q \in \mathcal{S}} \langle f\rangle_Q\mathbbm{1}_Q(x)=:\epsilon Sf(x)
$$
for almost every $x \in \supp f$. Since sparse operators are bounded on $L^p(w)$ with norm depending only on $n$, $[w]_{A_p}$, and the sparseness constant, one has
$$
	\|(T-T_N)f\|_{L^p(w)}\leq \epsilon\|Sf\|_{L^p(w)}\lesssim \epsilon\|f\|_{L^p(w)}
$$
and the result follows.
\end{proof}

We can now prove the weighted compactness result Theorem \ref{HaarWeightedCompactness}
\begin{proof}[Proof of Theorem \ref{HaarWeightedCompactness}]
Let $T_N$ be given by $T_Nf = \sum_{Q \in \mathcal{D}_N}\varepsilon_Q\langle f,h_Q\rangle h_Q$ and note that each $T_N$ is of finite rank and hence compact. Applying Theorem \ref{HaarSparseBound} gives that for every bounded $f$ with compact support, there exists a sparse collection $\mathcal{S}$ such that
$$
|(T-T_N)f(x)|= \bigg|\sum_{Q \in \mathcal{D}_N^c}\varepsilon_Q\langle f,h_Q\rangle h_Q\bigg|\lesssim \bigg(\sup_{Q \in \mathcal{D}_N^c}|\varepsilon_Q|\bigg)\sum_{Q\in\mathcal{S}}\langle |f|\rangle_Q\mathbbm{1}_Q(x)
$$
for almost every $x \in \supp f$. The result follows upon applying Theorem \ref{SparseImpliesWeightedCompactness}.
\end{proof}


\section{Calder\'on-Zygmund Operators}
\label{CZOs}

\subsection{Notation and definitions} 

In this section, all of our integrals, averages, pairings, etcetera will be taken with respect to Lebesgue measure on $\mathbb{R}^n$. We write $m$ for Lebesgue measure and denote the Lebesgue measure of a set $A \subseteq \mathbb{R}^n$ by $|A|$.

We consider three 
bounded functions 
satisfying 
\begin{equation}\label{limits}
\lim_{x\rightarrow \infty }L(x)=\lim_{x\rightarrow 0}S(x)=\lim_{x\rightarrow \infty }D(x)=0.
\end{equation}
Without loss of generality, we assume that $L$ and $D$ are non-increasing, while $S$ is non-decreasing. Moreover, since any dilation of a function satisfying a limit in (\ref{limits}) also satisfies the same limit, 
we omit universal constants appearing in the argument of these functions.

A measurable function $K:(\mathbb R^{n}\times \mathbb R^{n}) \setminus 
\{ (x,y)\in \mathbb R^{n}\times \mathbb R^{n} : x=y\} \to \mathbb C$ is a
\emph{compact Calder\'on-Zygmund kernel} if it is bounded on compact subsets of its domain
and there exist a function $\omega$ satisfying the Dini-type condition 
$$
\int_0^1\int_0^1\omega(st)\frac{ds}{s}\frac{dt}{t}<\infty 
$$
and bounded functions $L$, $S$, and $D$ satisfying \eqref{limits}
such that
\begin{equation}\label{smoothcompactCZ}
|K(x,y)-K(x',y')|
\lesssim
\omega\Big(\frac{|x-x'|+|y-y'|}{|x-y|}\Big)\frac{F_{K}(x,y)}{|x-y|^{n}},
\end{equation}
whenever $|x-x'|+|y-y'|\leq \frac{1}{2}|x-y|$ with
\begin{align}\label{LSDinF}
F_{K}(x,y)&
=L(|x-y|)S(|x-y|)D(|x+y|).
\end{align}
As shown in \cite{V2015}, inequality \eqref{smoothcompactCZ} and $\displaystyle {\lim_{|x-y|\rightarrow \infty} K(x,y)=0}$ imply that $K$ satisfies the following decay estimate
\begin{equation}\label{decaycompactCZ}
|K(x,y)|
\lesssim \frac{F_{K}'(x,y)}{|x-y|^{n}}
\end{equation}
whenever $x\neq y$, where $F_{K}'$ may be slightly different from the function in \eqref{LSDinF}, but it has a similar structure and it satisfies similar estimates.

For technical reasons, we will also use an alternative formulation of a compact Calder\'on-Zygmund kernel in which we substitute the function $F_K(x,y)$ of \eqref{smoothcompactCZ} 
with
\begin{equation}\label{DecaySub}
F_{K}(x,y,x',y')=L_{1}(|x-y|)S_{1}(|x-x'|+|y-y'|)D_{1}\Big(1+\frac{|x+y|}{1+|x-y|}\Big),
\end{equation}
where
$L_{1}$, $S_{1}$, and $D_{1}$ satisfy the limits in
(\ref{limits}). 
It was shown how this new condition can be obtained from \eqref{smoothcompactCZ} in \cite{V2015}. In general, we will omit the subindexes in the three factors of $F_K$, using the same notation as in \eqref{LSDinF}.

We work with \emph{Calder\'on-Zygmund operators} $T$ having compact extensions on $L^2(\mathbb{R}^n)$ and satisfying
\begin{equation}\label{kernelrep}
Tf(x) =\int_{\mathbb R^{n}} K(x,y)f(y)\, dy
\end{equation}
for compactly supported functions $f$ and $x \not \in \supp f$, where $K$ satisfies properties \eqref{smoothcompactCZ}, \eqref{decaycompactCZ}, and \eqref{DecaySub}
above.

Given two cubes $I, J \in \mathcal{D}$ with $\ell(I) \neq \ell(J)$, we denote the smaller of $I$ and $J$ by $I \wedge J$ and the larger of $I$ and $J$ by $I \vee J$. We define $\langle I,J\rangle$ to be the unique cube containing $I\cup J$ with the smallest possible side length and 
such that
$|c(I)|$ is minimum. Notice that $\langle I, J \rangle$ need not be dyadic. 
We also define the eccentricity and relative distance of $I$ and $J$ to be
$$
\ec(I,J):=\frac{\ell(I\wedge J)}{\ell(I\vee J)}\quad \text{and} \quad 
\rdist(I,J):=\frac{\ell(\langle I,J\rangle )}{\ell(I\vee J)}.
$$
Note that
\begin{align}\label{equivrdist}
\rdist(I,J) & \approx  1+\frac{|c(I)-c(J)|}{\ell(I)+\ell(J)}.
\end{align}
Given $I\in \mathcal D$, we denote the boundary of $I$ by $\partial I$,
and the inner boundary of $I$ by $\mathfrak{D}_{I}:=\displaystyle{\cup_{I'\in \child(I)}\partial I'}$.
When $J\subseteq 3I$, we define the inner relative distance of $J$ and $I$ by
$$
\inrdist(I,J):=1+\frac{\dist(J,{\mathfrak D}_{I})}{\ell(J)}.
$$

Given three cubes $I_{1},I_{2},$ and $I_{3}$, 
we denote
$$
F_{K}(I_{1}, I_{2}, I_{3}):=L(\ell(I_{1}))S(\ell(I_{2}))D(\hspace{-.03in}\rdist(I_{3},\mathbb B)) \quad \text{and} \quad F_{K}(I):=F(I,I,I),$$ where $\mathbb{B}:=[-\frac{1}{2},\frac{1}{2}]^n$. We define
\begin{align*}\label{Dtilde}
\tilde L(\ell(I)):=
\int_{0}^{1} \omega(t)L(\ell(t^{-1}I))\frac{dt}{t} \quad \text{and}\quad 
\tilde{D}(\hspace{-.03in}\rdist(I,\mathbb B)):=
\int_{0}^{1} W(t)D(\hspace{-.03in}\rdist(t^{-1}I,\mathbb B))\frac{dt}{t},
\end{align*}
where $W(t):=\int_0^t\omega(s)\frac{ds}{s}$. We also define the corresponding
$$
\tilde F_{K}(I_{1}, I_{2}, I_{3}):=\tilde L(\ell(I_{1}))S(\ell(I_{2}))\tilde D(\hspace{-.03in}\rdist(I_{3},\mathbb B)) \quad \text{and} \quad \tilde F_{K}(I):=\tilde F_{K}(I,I,I).$$


For a cube $Q$, let $Q^*$ be the cube such that $c(Q^*)=c(Q)$ and $\ell(Q^*)=5\ell(Q)$. For $Q \in \mathcal{D}$, we again write $h_Q$ for the Haar function adapted to $Q$, but now with respect to Lebesgue measure. Specifically, 
$$
    h_{Q}:=|Q|^{-\frac{1}{2}}(\mathbbm{1}_{Q}-2^{-n}\mathbbm{1}_{\widehat{Q}}).
$$
With this notation, $h_{Q}$ is supported on $\widehat{Q}$ and constant on $Q$ and on $\widehat{Q}\backslash Q$. 

Define the difference operator localized on a dyadic cube $Q$ as
\begin{equation}\label{Deltaincoord}
\Delta_{Q}f
:=\sum_{R\in \child(Q)}(\langle f\rangle_{R}-\langle f\rangle_{Q})\mathbbm{1}_{R},
\end{equation}
where now $\langle f\rangle_Q:= \frac{1}{|Q|}\int_{\mathbb{R}^n}f\, dm$. It is shown in \cite{V2019} that 
$$
\Delta_{Q}f
=\sum_{R\in \child(Q)}\langle f, h_{R}\rangle 
h_{R},
$$
where we write $\langle f, g\rangle:=\int_{\mathbb R^{n}} fg\, dm$. Thus by summing a telescopic sum, we get
\begin{align*}
\sum_{\substack{Q\in {\mathcal D}\\ 
2^{-N}\leq \ell(Q)\leq 2^N}}\sum_{R\in \child(Q)}\langle f, h_{R}\rangle h_R(x)
&=\sum_{\substack{Q\in {\mathcal D}\\ 
2^{-N}\leq \ell(Q)\leq 2^N}}\Delta_{Q}f(x)
=\langle f\rangle_{J}\mathbbm{1}_{J}(x)-\langle f\rangle_{I}\mathbbm{1}_{I}(x),
\end{align*}
where $I,J\in {\mathcal D}$ are such that $x\in J\subseteq I$, $\ell(J)=2^{-N}$ and $\ell(I)=2^{N+1}$.




We denote the wavelet father adapted to $Q$ as $\varphi_Q:=|Q|^{-1}\mathbbm{1}_{Q}$. 
Given a function $b\in \BMO$, 
the paraproduct operators associated with $b$ are defined as follows: 
$$
\Pi_{b}(f):=\sum_{I\in \mathcal D}\langle b, h_I\rangle
\langle f, \varphi_I\rangle h_I \quad \text{and} \quad \Pi^{*}_{b}(f):=\sum_{I\in \mathcal D}\langle b, h_I\rangle
\langle f,h_I\rangle \varphi_I.
$$
Note that the operator 
$$
\Pi^{*}_{P_M(b)}(f):=\sum_{I\in \mathcal D_M}\langle b, h_I\rangle
\langle f,h_I\rangle \varphi_I
$$
is of finite rank.



A linear operator $T$ 
satisfies the weak compactness condition
if there exists a bounded function $F_{W}$ satisfying 
$
{\displaystyle \lim_{\ell(Q)\rightarrow \infty }F_W(Q)
=\lim_{\ell(Q)\rightarrow 0 }F_W(Q)
=\lim_{c(Q)\rightarrow \infty }F_W(Q)=0}
$
such that 
\begin{equation}\label{restrictcompact2}
|\langle T\mathbbm{1}_{Q},\mathbbm{1}_{Q}\rangle |
\lesssim 
|Q|F_{W}(Q)
\end{equation}
for all $Q\in {\mathcal D}$. 

We define $\CMO(\mathbb R^{n})$ as the closure in $\BMO(\mathbb R^{n})$ of the space of continuous functions vanishing at infinity.

For positive integers $N$, define the \emph{projection operator} $P_N$ on lagom cubes by $$P_Nf:=\sum_{Q\in \mathcal{D}_N}\langle f,h_Q \rangle h_Q$$ 
and also
$$
    P_N^{\perp}f:= (I-P_N)f=\sum_{Q \in \mathcal{D}_N^c}\langle f,h_Q\rangle h_Q
$$ 
with convergence interpreted pointwise almost everywhere.

\begin{rem}
 To show that a linear operator
$T$ is compact on $L^2(\mathbb R^n)$, for instance, one can equivalently show that for every $\varepsilon>0$, there exists $N_0>0$ so that 
$$
    \|P_N^{\perp}Tf\|_{L^2(\mathbb R^n)}\lesssim \varepsilon\|f\|_{L^2(\mathbb R^n)}
$$
for all $N>N_0$ and all $f\in L^2(\mathbb R^n)$. 
\end{rem}

\subsection{Technical results}
\label{tech}
The following result is proved in \cite{V2019} in the particular case of $\omega(t)=t^{\delta}$ with $0<\delta\leq 1$.  The proof of this lemma is a straightforward modification of that contained in \cite{V2019}*{Proposition 8.2}.
\begin{lemma}
\label{tildeT}
Let $T$ be a linear operator associated with a compact Calder\'on-Zygmund kernel satisfying the weak compactness condition 
\eqref{restrictcompact2} and such that $T1, T^*1\in \CMO$.
If $\tilde T:=T-\Pi^{*}_{P_N(T^*1)}$, then
$$
	\|P_N^{\perp}\tilde T f\|_{L^{1,\infty}(\mathbb{R}^n)}\lesssim \sup_{Q \in \mathcal{D}_{N}^c} \varepsilon_Q \hskip5pt \|f\|_{L^1(\mathbb{R}^n)}
$$
for all $f \in L^1(\mathbb{R}^n)$. The coefficients $\varepsilon_Q$ are defined for $Q\in \mathcal{D}_N^c$ by
\begin{align*}
\varepsilon_Q:&=
\sum_{\substack{e\in \mathbb Z, m\in \mathbb N }}\omega(2^{-|e|})\frac{\omega(m^{-1})}{m}
\max_{\substack{R \in \mathcal{D}\\R\in Q_{e,m}}}\max_{i=1,2,3}F_i(Q,R), 
\end{align*}
where 
$
Q_{e,m}:=\{ R\in {\mathcal D}:\ell(Q)=2^{e}\ell(R) \quad\text{and}\quad m\leq  \rdist(Q,R)< m+1 \},
$
$\mathcal{D}_N^c(Q):=\mathcal {D}(Q)\cap \mathcal{D}_N^c$,
and
\begin{itemize}
\item[i)] when $\rdist (Q,R)> 3$,
$$
F_1(Q,R):=F_{K}(\langle Q,R\rangle ,Q\wedge R, \langle Q,R\rangle ),
$$

\item[ii)] 
$\rdist (Q,R)\leq 3$  
and $\inrdist(Q,R)>1$, 
$$
F_{2}(Q, R):=\tilde F_{K}( Q\wedge R ,Q\wedge R, \langle Q,R\rangle ),
$$
\item[iii)] and when $\rdist (Q,R)\leq 3$ and $\inrdist(Q,R)=1$, 
\begin{align*}
F_{3}(Q, R):=F_{2}(Q,R)+\tilde F_{K}(Q\wedge R) +\delta(Q,R)F_{W}(Q)
\end{align*}
with $\delta(Q,R)=1$ if $Q=R$ and zero otherwise, while
$$
F_{W}(Q)=\sup_{Q}\langle |T1-\langle T1\rangle_{Q}|\rangle_{Q}
+\langle |T^*1-\langle T^*1\rangle_{Q}|\rangle_{Q}, 
$$ 
and the supremum is taken over all cubes with sides parallel to the coordinate axes. 
\end{itemize}
\end{lemma}
\begin{rem}
We note that the definition of $\varepsilon_{Q}$
in Lemma \ref{tildeT}
follows from the reasoning used in 
\cite{V2019}*{Proposition 8.2} after incorporating the Dini-type function $w$, which in that paper is just the classical $w(t)=t^\delta$.

The expression $\varepsilon_Q$ 
gathers the contributions of all terms in the wavelet decomposition of the argument functions. As deduced from the proof of \cite{V2019}*{Proposition 8.2}, 
all these terms are positive and decrease with eccentricity $e$ and relative distance $m$ so that the contribution of all terms 
$\tilde F_K(I_1,I_2,I_3)$ with distinct cubes is comparable to the contribution of the corresponding terms with equal cubes. 
Then one can see that  
$\tilde F_K(Q)\lesssim \varepsilon_Q $
and 
${\displaystyle \lim_{N\rightarrow \infty} \sup_{Q \in \mathcal{D}_{N}^c} \varepsilon_Q =0}$.

 Finally, we also note that $Q\in \mathcal{D}_N^c$ with $\ell(Q)\leq 2^{-N}$ implies ${\mathcal D}(Q)\subseteq {\mathcal D}_N^c$, and also that
 $$
    \sup_{Q}\langle |T1-\langle T1\rangle_{Q}|\rangle_{Q}\approx \bigg(|Q|^{-1}\sum_{R\in {\mathcal D}_N^c(Q)}\langle T1,h_R\rangle^2
\bigg)^{\frac{1}{2}} \leq \| P_{N}^\perp (T1)\|_{\BMO}.
$$
\end{rem}


\begin{lemma}\label{paraT}
If $T$ is a linear operator associated to a compact Calder\'on-Zygmund kernel, $Q \in \mathcal D$, and $N>1$, then
\begin{align*}
|P_N^\perp T(f\mathbbm{1}_{\mathbb{R}^n\setminus Q^*})(x)-P_N^\perp T(f\mathbbm{1}_{\mathbb{R}^n\setminus Q^*})(x')|
\leq \bar\varepsilon_{Q}Mf(x),
\end{align*}
for all $f \in L^1(\mathbb{R}^n)$ and all  $x, x' \in Q$, where 
$\bar \varepsilon_Q := L(\ell(Q))S(\ell(Q))\tilde{ D}(\hspace{-.03in}\rdist(Q, \mathbb B))\leq \tilde F_K(Q)\leq \varepsilon_Q $ with 
$\varepsilon_Q$ as in Lemma \ref{tildeT}.
\end{lemma}

\begin{proof} 
By definition
\begin{align}\label{PTout}
|P_N^\perp T(f\mathbbm{1}_{\mathbb{R}^n\setminus Q^*})(x)&-P_N^\perp T(f\mathbbm{1}_{\mathbb{R}^n\setminus Q^*})(x')|
\leq \sum_{R\in {\mathcal D}_N^c}|
\langle T(f\mathbbm{1}_{\mathbb{R}^n\setminus Q^*}), h_R\rangle | |h_R(x)-h_R(x')|.
\end{align}
For $R\in {\mathcal D}$ such that $\widehat{R}\cap Q=\emptyset$, we have $h_R(x)=h_R(x')=0$, 
while if $Q\subsetneq \widehat{R}$ we have 
$h_R(x)=h_R(x')$, and so the corresponding terms in \eqref{PTout} are zero. 
On the other hand, for $\widehat{R}\subseteq Q$, 
we have that $h_R(x)-h_R(x')\neq 0$ implies $x\in \widehat{R}$ or $x'\in \widehat{R}$. Moreover, in that case we have $|h_R(x)-h_R(x')|\lesssim |R|^{-\frac{1}{2}}$. 

Now, since  $\widehat{R}\subseteq Q$ implies that $\widehat{R}$ does not intersect $\mathbb R^n\setminus Q^*$, 
we can use the integral representation of $T$ and the mean zero property of $h_R$ to write
\begin{align*}
\langle T(f\mathbbm{1}_{\mathbb{R}^n\setminus Q^*}),h_R\rangle
&=\int_{\widehat{R}}\int_{\mathbb{R}^n\setminus Q^*}f(y)h_R(z)(K(z,y)-K(c(\widehat{R}),y))dy dz.
\end{align*}
Since $|x-x'|\leq \ell(Q)\leq \frac{1}{2}|x-y|$ for all $y\in \mathbb{R}^n\setminus Q^*$, we can use the smoothness condition of the kernel to write
\begin{align*}
|\langle T(f\mathbbm{1}_{\mathbb{R}^n\setminus Q^*}),h_R\rangle|
&\leq \int_{\widehat{R}}\int_{\mathbb{R}^n\setminus Q^*}|f(y)||h_R(z)||K(z,y)-K(c(\widehat{R}),y)|\,dy dz\\
&\leq \int_{\widehat{R}} |h_R(z)|\sum_{k=0}^{\infty} \int_{2^{k+1}Q^*\setminus 2^kQ^*}
\omega\Big(\frac{|z-c(\widehat{R})|}{|z-y|}\Big)
\frac{F_K(z,c(\widehat{R}),y)}{|z-y|^{n}}|f(y)|\,dydz,
\end{align*}
where
\begin{align*}
F_{K}(z,c(\widehat{R}),y):= L(|z-y|)S(|z-c(\widehat{R})|)D\Big(1+\frac{|z+y|}{1+|z-y|}\Big).
\end{align*}
Since $\ell(Q)\leq 2^{k-1}\ell(Q^*)\leq |z-y|$ and $|z-c(\widehat{R})|\leq \ell(\widehat{R})/2=\ell(R)\leq \ell (Q)$, we have
$L(|z-y|)\leq L(\ell(Q))$ and $S(|z-c(\widehat{R})|)\leq S(\ell(Q))$.

To deal with $L$, we first note that
$$2^{k}\ell(Q)\leq 2^{k-1}\ell(Q^*)\leq |z-y|\leq 2^{k}\ell(Q^*)= 2^{k}5\ell(Q),$$ 
that is, $|z-y|\approx 2^{k}\ell(Q)$.
Using this and 
$|z|\leq \frac{1}{2}(|z-y|+|z+y|)$, we have
$$
1+\frac{|z|}{1+2^{k}\ell(Q)}
\lesssim 1+\frac{|z|}{1+|z-y|}
\leq \frac{3}{2}\Big(1+\frac{|z+y|}{1+|z-y|}\Big).
$$
Moreover, 
since $|z-c(Q)|\leq \ell(Q)/2$, we also have 
$1+\frac{|c(Q)|}{1+2^k\ell(Q)}\leq \frac{5}{4}\big(1+\frac{|z|}{1+2^k\ell(Q)}\big)$. Using this and \eqref{equivrdist}, we have
\begin{align*}
1+\frac{|z|}{1+2^{k}\ell(Q)}
&\gtrsim 1+\frac{|c(2^{k} Q)|}{1+2^{k}\ell(Q)}
\gtrsim \rdist (2^{k} Q,\mathbb B).
\end{align*}
Then 
\begin{align*}
F_{K}(z,c(\widehat{R}),y)
&\leq L(\ell(Q))
S(\ell(Q))D(\hspace{-.03in}\rdist (2^k Q,\mathbb B))
=F_{K}(Q,Q,2^k Q).
\end{align*}
Using previous estimates together with the facts that $|z-c(\widehat{R})|\leq \ell(\widehat{R})$, $2^k\ell(Q)\lesssim |z-y|$, and 
$\| h_R\|_{L^1(\mathbb R^n)}\lesssim |R|^{\frac{1}{2}}$, 
we get
\begin{align*}
|\langle T(f\mathbbm{1}_{\mathbb{R}^n\setminus Q^*}),h_R\rangle|
&\lesssim L(\ell(Q))S(\ell(Q))
\int_{\widehat{R}} |h_R(z)|dz
\\
&\hskip50pt \sum_{k=0}^{\infty} \omega\Big(\frac{\ell(\widehat{R})}{2^k\ell(Q)}\Big) 
D(\hspace{-.03in}\rdist (2^{k} Q,\mathbb B)) \frac{1}{|2^{k+1}Q^*|}\int_{2^{k+1}Q^*} |f(y)| dy\\
&\lesssim L(\ell(Q))S(\ell(Q))|R|^{\frac{1}{2}}
\sum_{k=0}^{\infty} \omega\Big(2^{-k}\frac{\ell(\widehat{R})}{\ell(Q)}\Big) 
D(\hspace{-.03in}\rdist (2^{k} Q,\mathbb B)) Mf(x)\\
&\lesssim |R|^{\frac{1}{2}}L(\ell(Q))S(\ell(Q))\int_{0}^{1} \omega\Big(t\frac{\ell(\widehat{R})}{\ell(Q)}\Big)D(\hspace{-.03in}\rdist(t^{-1}Q,\mathbb B))\frac{dt}{t}Mf(x).
\end{align*}

Now we parametrize all dyadic cubes $R\in {\mathcal D}_N^c$ such that $\widehat{R}\subseteq Q$ and $x\in \widehat{R}$ or $x'\in \widehat{R}$ by length $\ell(R_j)=2^{-j}\ell(Q)$. We note that there are at most two such cubes for each fixed $j$, one containing $x$ and another one containing $x'$. By summing over all these cubes, we finally get 
\begin{align*}
|P_N^\perp T(f\mathbbm{1}_{\mathbb{R}^n\setminus Q^*})(x)&-P_N^\perp T(f\mathbbm{1}_{\mathbb{R}^n\setminus Q^*})(x')|
\lesssim \sum_{j=0}^{\infty }
|\langle T(f\mathbbm{1}_{\mathbb{R}^n\setminus Q^*}), h_{R_j}\rangle | |R_j|^{-\frac{1}{2}}
\\
&\lesssim L(\ell(Q))S(\ell(Q))\sum_{j=0}^\infty\int_{0}^{1}\omega(t2^{-j})D(\hspace{-.03in}\rdist(t^{-1}Q,\mathbb B))\frac{dt}{t}Mf(x)
\\
&\lesssim L(\ell(Q))S(\ell(Q))\int_{0}^{1}\int_{0}^{1} \omega(ts)D(\hspace{-.03in}\rdist(t^{-1}Q,\mathbb B))\frac{dt}{t}\frac{ds}{s}Mf(x)
\\
&\leq \bar \varepsilon_{Q}Mf(x).
\end{align*}
\end{proof}

We can also prove the following result using similar ideas. 
\begin{cor}
\label{ParaproductDifferenceBound}
If $b \in \CMO(\mathbb R^n)$, $Q \in \mathcal D$, and $N>1$, then
\begin{align*}
|P_N^{\perp}\Pi^*_b(f\mathbbm{1}_{\mathbb{R}^n\setminus Q^*})(x)-P_N^{\perp}\Pi^*_b(f\mathbbm{1}_{\mathbb{R}^n\setminus Q^*})(x')|
\leq \bar\varepsilon_{Q}Mf(x),
\end{align*}
for all $f \in L^1(\mathbb{R}^n)$ and all  $x, x' \in Q$, where 
$\bar \varepsilon_Q := L(\ell(Q))S(\ell(Q))\tilde D(\hspace{-.03in}\rdist(Q, \mathbb B))\leq \varepsilon_Q $
with $L(t)=\| P_{{\mathcal D}_{t}^{c}} b\|_{\BMO}$, 
$S(t)=\| P_{{\mathcal D}_{t^{-2/3}}^{c}} b\|_{\BMO}
+(1+\| b\|_{\BMO})^{\frac{1}{2}}(\frac{t}{1+t})^{\frac{2}{3}}$, 
and $D(t)=\| P_{{\mathcal D}_{\log t}^{c}} b\|_{\BMO}$.
We note that, using the ceiling function notation, $\mathcal D_{t}$ actually denotes $\mathcal D_{\lceil{t}\rceil}$.
\end{cor}
\begin{proof}

It was shown in \cite{V2015} that if $b \in \CMO$, then the paraproduct operator $\Pi_b^*$ is associated to a compact Calder\'on-Zygmund kernel with constant given by 
\begin{align*}
&\| P_{{\mathcal D}_{|x-y|}^{c}} b\|_{\BMO}
(\| P_{{\mathcal D}_{|x-y|^{-2/3}}^{c}} b\|_{\BMO}
+(1+\| b\|_{\BMO})^{\frac{1}{2}}\min(1,|x-y|^{\frac{2}{3}}))
\| P_{{\mathcal D}_{\log|x+y|}^{c}} b\|_{\BMO}
\\
&=L(|x-y|)S(|x-y|)D(|x+y|).
\end{align*}
Similar reasoning to that developed in Lemma \ref{paraT} yields the result.
\end{proof}


As seen in \cite{PPV2017}, $P_N$ is bounded on $\CMO$. Then
the hypothesis $T^*1\in \CMO$ implies 
that also $P_NT^*1 \in \CMO$. In fact, 
$\| P_NT^*1\|_{\BMO}
\leq \| T^*1\|_{\BMO}$
and 
$\| P_M^\perp(P_NT^*1)\|_{\BMO}=0$ for all $M>N$. 
This justifies the expression  
$\Pi^{*}_{P_N(T^*1)}$ in the next result, which follows from Lemma \ref{paraT} and Corollary \ref{ParaproductDifferenceBound}. 
\begin{cor}
\label{paratildeT}
If $T$ is a linear operator associated with a compact Calder\'on-Zygmund kernel, $\tilde T:=T-\Pi^{*}_{P_N(T^*1)}$, and $Q \in \mathcal{D}$, then
\begin{align*}
|P_N^\perp \tilde T(f\mathbbm{1}_{\mathbb{R}^n\setminus Q^*})(x)-P_N^\perp \tilde T(f\mathbbm{1}_{\mathbb{R}^n\setminus Q^*})(x')|
\leq \bar\varepsilon_{Q}Mf(x)
\end{align*}
 for all $f \in L^1(\mathbb{R}^n)$ and all  $x, x' \in Q$, where $\bar \varepsilon_Q := L(\ell(Q))S(\ell(Q))\tilde D(\rdist(Q, \mathbb B))\leq \varepsilon_Q $. 
\end{cor}

A consequence of the work in Lemma \ref{paraT} and Corollary \ref{ParaproductDifferenceBound} is that 
the kernels of both $T$ and $\Pi^{*}_{P_N(T^*1)}$ share similar estimates. In the next section we denote by $K$ the kernel of $\tilde T$, which satisfies the properties of a compact Calder\'on-Zygmund kernel \eqref{smoothcompactCZ}, \eqref{decaycompactCZ}, and \eqref{DecaySub}. 

\subsection{Sparse domination for compact Calder\'on-Zygmund operators}

\begin{thm}
\label{domtilde}
Let $T$ be a linear operator associated to a compact Calder\'on-Zygmund kernel satisfying the weak compactness condition 
\eqref{restrictcompact2} and $T1, T^*1\in \CMO$ and let $\tilde T := T-\Pi^{*}_{P_N(T^*1)}$. 
For every
$\varepsilon >0$ there exists $N_0>0$ such that for all $N>N_0$ and 
every compactly supported $f \in L^1(\mathbb{R}^n)$, there is  
a sparse family of cubes ${\mathcal S}$ such that  
$$
	|P_N^\perp \tilde Tf(x)|\lesssim \varepsilon \sum_{R \in \mathcal{S}}  \langle |f| \rangle_{R^*} \mathbbm{1}_R(x) =: \varepsilon S|f|(x)
$$
for almost every $x \in \mathbb{R}^n$.
\end{thm}

\begin{proof} 
Without loss of generality, suppose there is a dyadic cube $B$ such that $\ell(B)>1$ and $\supp f\subseteq B$. Let $\{\varepsilon_Q\}_{Q\in {\mathcal D}}$ be the sequence in the statement of Lemma \ref{tildeT} which satisfies 
$\displaystyle\lim_{N\rightarrow \infty} \sup_{Q \in \mathcal{D}_{N}^c} \varepsilon_Q =0$.
Given $\varepsilon>0$, let $N_0>0$ be such that $\displaystyle\sup_{Q\in {\mathcal D}_{N}^c} \varepsilon_Q <\varepsilon $ for all $N>N_0$.
Fix $N>N_0$ and let $Q_{0}$ be a cube such that $B\subsetneq Q_0$ and $\dist(B, Q_0^c)\geq 2^{N+3}\ell(B)$.

We first establish the sparse estimate outside of $Q_0$. 
For $j\geq 0$, we write $Q_j:=2^jQ_0$ and for $j\geq 1$ we define $P_j:=Q_{j}\backslash Q_{j-1}$. 
Note that the family $\{ Q_j\}_{j\geq 0}$ is sparse by construction.

Let $x\in P_j$. By definition,
\begin{align}\label{outside}
P_N^\perp \tilde T f(x)
	&=\tilde T f(x)-\sum_{R\in {\mathcal D}_N}\langle \tilde Tf,h_R\rangle h_R(x);
\end{align}
we will bound each term separately. For the first term, since $x\notin \supp f$, we can write
\begin{align*}
	|\tilde Tf(x)|&
	= \bigg|\int_{B}K(x,y)f(y)\,dy\bigg|
	\leq 
	\int_{B}\frac{F_K(x,y)}{|x-y|^n}|f(y)|\,dy
\end{align*}
where $K$ denotes the kernel of $\tilde T$ and  
$$
F_K(x,y)=L(|x-y|)S(|x-y|)D\Big(1+\frac{|x+y|}{1+|x-y|}\Big).
$$
Since 
$|x-y|\approx \ell(Q_{j})$ and 
$|y-c(Q_0)|\lesssim \ell(Q_0)$ for $y \in B$, we have by the same reasoning used in Lemma \ref{paraT} that
$L(|x-y|)\leq L(\ell(Q_j))$, 
$S(|x-y|)\leq S(\ell(Q_j))$, and
$D(1+\frac{|x+y|}{1+|x-y|})\lesssim D(\hspace{-.03in}\rdist (Q_j,\mathbb B))$.
Then 
\begin{align}\label{Foutside}
F_{K}(x,y)
&\leq L(\ell(Q_j))
S(\ell(Q_j))D(\hspace{-.03in}\rdist (Q_j,\mathbb B))
=F_{K}(Q_j)\leq \varepsilon,
\end{align}
where in the last inequality we used that $\ell(Q_j)\geq \ell(Q_0)>2^{N+3}\ell(B)\geq 2^{N+3}$, and thus $Q_j\in {\mathcal D}_N^c$. With this and the fact that 
$|x-y|\approx \ell(Q_j)$, we have
\begin{align*}
	|\tilde Tf(x)|&\lesssim \frac{\varepsilon }{|Q_j|}\|f\|_{L^1(\mathbb{R}^n)}
	= \varepsilon \langle |f| \rangle_{Q_j}\mathbbm{1}_{Q_{j}}(x).
\end{align*}

On the other hand, we note that the second term in \eqref{outside} is defined by a telescopic sum  
such that, for fixed $x\in P_{j}$, the collection of cubes $R\in {\mathcal D}_N$ with  
$x\in \widehat{R}$ form a convex chain. With this we mean that if $R_1, R_2\in \mathcal D_N$ with $R_1\subseteq R_2$, then any other cube $R'\in \mathcal D$ with $R_1\subseteq R' \subseteq R_2$ satisfies $R'\in \mathcal D_N$ and so, $R'$ is also in the sum. Moreover, we can obviously assume that the sum is non-empty. In that case, we have
\begin{align*}
\sum_{R\in {\mathcal D}_N}\langle \tilde Tf,h_R\rangle h_R(x)
&
=
\langle \tilde Tf\rangle_{J}\mathbbm{1}_{J}(x)-
\langle \tilde Tf\rangle_{I}\mathbbm{1}_{I}(x),
\end{align*}
where $I,J\in \mathcal D$ are such that 
$x\in J\subseteq I$, 
and $\ell(J),\ell(I)\leq 2^{N+1}$. 
We can now apply the same ideas to bound both terms, and so we only write the estimates for the second term. Since $x\in I\cap P_j$, we have as before
$2^{j-2}\ell(Q_0)<|x-y|$
for all $y\in B\subseteq Q_0$.
On the other hand, 
$$
2^{j-2}\ell(Q_0)\geq \ell(Q_0)/2\geq 
2^{N+2}\ell(B)
\geq 2\ell(I),
$$
which implies $\ell(I)\leq 2^{j-3}\ell(Q_0)$.
With this and 
$|t-x|\leq \ell(I)$ for all 
$t\in I$, we get
$$
|t-y|\geq |x-y|-|t-x|
\geq 2^{j-2}\ell(Q_0)
-\ell(I)
\geq 2^{j-3}\ell(Q_0)
$$
and 
$$
|t-y|\leq |x-y|+|t-x|
\lesssim 2^{j}\ell(Q_0)
+\ell(I)
\leq 2^{j+1}\ell(Q_0).
$$
Therefore, 
$
|t-y|
\approx \ell(Q_j)
$. We also have $|y-c(Q_0)|\leq \ell(Q_0)/2$. 
Write
\begin{align*}
|\langle \tilde Tf\rangle_{I}\mathbbm{1}_{I}(x)|&= \frac{1}{|I|}\Big|\int_{I}\int_{B}K(t,y)f(y)\,dy\, dt\Big|
\\
&
\lesssim \frac{1}{|I|}\int_{I}\int_{B}\frac{F_K(t,y)}{|x-y|^n}|f(y)|\,dy dt
\\
&\lesssim \frac{F_K(Q_j)}{|Q_j|}\|f\|_{L^1(\mathbb{R}^n)}
\leq \varepsilon \langle |f| \rangle_{Q_j}\mathbbm{1}_{Q_{j}}(x),
\end{align*}
where the last inequality follows from  
\eqref{Foutside}.

We now work to establish the sparse bound inside $Q_0$. 
For this local piece, we follow the ideas from \cite{LO2019}
to define recursively the desired sparse family $\mathcal S$ and sparse operator $S$. 
Let $\mathcal{D}_N^c(Q_0):=\mathcal{D}(Q_0)\cap \mathcal{D}_N^c$ and 
$\mathcal Q:=\left\{Q \in \mathcal{D}_N^c(Q_0) : \ell(Q)=2^{-(N+2)}\right\}$. We decompose $\tilde Tf$ as 
$$
    \tilde Tf=\sum_{Q\in \mathcal Q}\tilde T(f\mathbbm{1}_{Q}).
$$
If we assume the desired sparse domination result holds for $P_N^{\perp}\tilde T(f\mathbbm{1}_{Q})$, then by disjointness of the cubes $Q$, we can deduce a similar sparse estimate for $P_N^{\perp}\tilde T$:
\begin{align*}
    |P_N^{\perp}\tilde Tf|&\leq \sum_{Q\in \mathcal Q}|P_N^{\perp}\tilde T(f\mathbbm{1}_{Q})|
    \lesssim \varepsilon \sum_{Q\in \mathcal Q}S|f\mathbbm{1}_{Q}|
\\
&
= \varepsilon \sum_{Q\in \mathcal Q} \sum_{R\in {\mathcal S}(Q)}\langle |f|\rangle_{R}
\mathbbm{1}_{R}
    \leq \varepsilon \sum_{R\in {\mathcal S}(Q_{0})}\langle |f|\rangle_{R}
\mathbbm{1}_{R}.
\end{align*}
Therefore, we will only prove the sparse estimate for each $P_N^{\perp}\tilde T(f\mathbbm{1}_{Q})$.

We start by adding all cubes $Q\in \mathcal Q$ to the family $\mathcal S$ and functions $\langle |f|\rangle_{Q}\mathbbm{1}_{Q}$ to the sparse operator $S|f|$. These cubes are pairwise disjoint and satisfy $\sum_{Q\in \mathcal Q}|Q|= |Q_{0}|$. This family does not satisfy the sparseness condition, 
but we can divide the family into two disjoint subfamilies ${\mathcal Q}_1$, ${\mathcal Q}_2$ containing exactly half of the cubes, each satisfying the sparseness condition $\sum_{Q\in {\mathcal Q}_i}|Q|= |Q_{0}|/2$. This leads to a domination by at most two sparse operators, which is acceptable. To simplify notation, we still denote each subfamily by $\mathcal{Q}$.

Fix $Q\in \mathcal Q$ and define
\begin{equation}\label{EQ}
E_{Q}:=\{x\in Q : M(f\mathbbm{1}_{Q})(x)> c'\langle |f|\rangle_{Q}\}
	\cup \{x\in Q : |P_N^\perp \tilde T(f\mathbbm{1}_{Q})(x)|> c'\varepsilon \langle |f|\rangle_{Q}\}, 
\end{equation}
where $c'>0$ is chosen so that 
$$
	|E_{Q}|\leq \frac{1}{2^{n+2}}|Q|.
$$
To show that $c'>0$ is independent of $\varepsilon$, from Lemma \ref{tildeT} we have
\begin{align*}
|E_{Q}|&\leq \frac{C}{c'\langle |f|\rangle_Q}\| f\mathbbm{1}_{Q}\|_{L^1(\mathbb R^n)}
+\frac{C\sup_{Q\in {\mathcal D}_N^c}|\varepsilon_Q|}{c'\varepsilon \langle |f|\rangle_Q}\| f\mathbbm{1}_{Q}\|_{L^1(\mathbb R^n)}
\leq \frac{2C}{c'}|Q|\leq \frac{1}{2^{n+2}}|Q|
\end{align*}
by choosing $c'>C2^{n+3}$.
We note that the constant $C>0$ may depend on the dimension $n$ but not on $\varepsilon >0$.

We define another exceptional set
$$
	\tilde E_{Q}:=\{x\in Q : M(\mathbbm{1}_{E_{Q}})(x)> 2^{-(n+1)}\},
$$
and define $\mathcal E_{Q}$ to be the family of maximal (with respect to inclusion) dyadic cubes $P$ contained in $\tilde E_{Q}$. 
Note that for each $Q\in \mathcal Q$, the containment $\mathcal{E}_{Q}\subseteq \mathcal{D}_N^c$ holds. Moreover, 
due to maximality, the cubes $P\in \mathcal E_{Q}$ are pairwise disjoint, and thus $\mathcal E_{Q}$ is a sparse collection:
\begin{equation}\label{Psparse}
\sum_{P\in \mathcal E_{Q}}|P|\leq |E_{Q}|\leq \frac{1}{2^{n+2}}|Q|.
\end{equation}
We see now that 
\begin{align}\label{PEQ}
	|P \cap E_{Q}| \leq \frac{1}{2}|P|.
\end{align}
By maximality, $2P\cap (Q\backslash E_Q)\neq \emptyset $, and so there exists 
$x\in 2P$ such that $M(\mathbbm{1}_{E_{Q}})(x)\leq 2^{-(n+1)}$. Then
$$
\frac{|E_{Q}\cap 2P|}{|2P|}\leq 2^{-(n+1)},
$$
which proves the upper inequality. 
Note that the 
inequality in \eqref{PEQ} implies $|P \setminus E_{Q}| > \frac{1}{2}|P|$. We can now estimate $|P_N^\perp \tilde T(f\mathbbm{1}_Q)(x)|$ for $x \in Q$.
First, 
for $x\in Q\backslash 
E_{Q}$ we trivially have 
$$
|P_N^\perp \tilde T(f\mathbbm{1}_{Q})(x)|\leq  c'\varepsilon \langle |f|\rangle_{Q}
=c'\varepsilon \langle |f|\rangle_{Q}
\mathbbm{1}_{Q}(x).
$$
Second, to obtain an estimate on $E_Q$, 
we note that $\left|E_{Q} \setminus \bigcup_{P\in \mathcal E_{Q}}P\right|
\leq \left|\tilde E_{Q} \setminus \bigcup_{P\in \mathcal E_{Q}}P\right|=0$, 
and so, we do not need to bound $|P_N^\perp \tilde T(f\mathbbm{1}_Q)(x)|$ for $x \in E_Q \setminus \bigcup_{P\in \mathcal E_{Q}}P$.

It only remains to control $|P_N^\perp \tilde T(f\mathbbm{1}_Q)(x)|$ for $x \in \bigcup_{P\in \mathcal E_{Q}}P$. For any $P\in \mathcal E_{Q}$, any $x \in P$, and any $x' \in P\setminus E_Q$, we decompose $P_N^\perp \tilde Tf(x)$ as follows: 
\begin{align*}
	|P_N^\perp \tilde T(f\mathbbm{1}_Q)&(x)|\leq |P_N^\perp \tilde T(f\mathbbm{1}_{Q\setminus P^*})(x)|+|P_N^\perp \tilde T(f\mathbbm{1}_{P^*})(x)|\\
	&\leq |P_N^\perp \tilde T(f\mathbbm{1}_{Q\setminus P^*})(x)-P_N^\perp \tilde T(f\mathbbm{1}_{Q\setminus P^*})(x')|+|P_N^\perp \tilde T(f\mathbbm{1}_{Q\setminus P^*})(x')|
	\\
	&\hskip50pt +|P_N^\perp \tilde T(f\mathbbm{1}_{P^*})(x)|\\
	&\leq |P_N^\perp \tilde T(f\mathbbm{1}_{Q\setminus P^*})(x)-P_N^\perp \tilde T(f\mathbbm{1}_{Q\setminus P^*})(x')|+|P_N^\perp \tilde T(f\mathbbm{1}_Q)(x')|+|P_N^\perp \tilde T(f\mathbbm{1}_{P^*})(x')|
	\\
	&\hskip50pt +|P_N^\perp \tilde T(f\mathbbm{1}_{P^*})(x)|\\
	&:= \text{I} + \text{II} + \text{III} + \text{IV}.
\end{align*}
The second term is easily controlled since $x' \not \in E_{Q}$
implies $|P_N^\perp \tilde T(f\mathbbm{1}_Q)(x')|\leq c' \varepsilon \langle |f| \rangle_{Q}$, and so
$$
	\text{II} \leq c' \varepsilon \langle |f| \rangle_{Q}\mathbbm{1}_{Q}(x).
$$

For the first and third terms, define 
$$
E_{P}':=
	\{x \in P: |P_N^\perp \tilde T(f\mathbbm{1}_{P^*})(x)|>c' \varepsilon \langle |f| \rangle_{P^*}\}.
$$
By Lemma \ref{tildeT}, 
$$	
|E_{P}'|
	\leq \frac{C\varepsilon}{c' \varepsilon \langle |f| \rangle_{P^*}}\|f\mathbbm{1}_{P^*}\|_{L^1(\mathbb{R}^n)}
	\leq \frac{1}{2^{n+2}}|P|.
$$
Then $|P \setminus E_{P}'| > \frac{1}{2}|P|$. This, together with 
$|P \setminus E_{Q}| > \frac{1}{2}|P|$, implies that 
$(P \setminus E_{Q}) \cap (P \setminus E_{P}')\neq \emptyset $. Therefore, there exists $x'\in P$ such that $M(f\mathbbm{1}_Q)(x') \leq c'\langle |f|\rangle_{Q}$ and
$|P_N^\perp \tilde T(f\mathbbm{1}_{P^*})(x')|\leq c' \varepsilon \langle |f|\rangle_{P^*}$. Then, since $(f\mathbbm{1}_Q)\mathbbm{1}_{\mathbb R^n \setminus P^*}=
f\mathbbm{1}_{Q\setminus P^*}$, we can apply Corollary \ref{paratildeT} to obtain
$$
	\text{I}\leq \varepsilon M(f\mathbbm{1}_Q)(x') \leq c'\varepsilon \langle |f|\rangle_{Q}\mathbbm{1}_{Q}(x).
$$
Moreover,
$$
	\text{III}\leq c' \varepsilon \langle |f|\rangle_{P^*}
	= c'\varepsilon \langle |f|\rangle_{P^*}\mathbbm{1}_{P}(x).
$$
We add the cubes $P\in \mathcal{E}_Q$ to the family $\mathcal S$ and the functions $\langle |f|\rangle_{P^*}\mathbbm{1}_{P}$ into $S|f|$. The family $\mathcal{E}_Q$ is sparse by \eqref{Psparse}. 

The fourth term is controlled by iterating the above argument, starting at \eqref{EQ} but replacing $Q$ with $P$, and so defining 
\begin{equation*}
E_{P}:=\{x\in P : M(f\mathbbm{1}_{P})(x)> c'\langle |f|\rangle_{P}\}
	\cup \{x\in P : |P_N^\perp \tilde T(f\mathbbm{1}_{P})(x)|> c'\varepsilon \langle |f|\rangle_{P}\}. 
\end{equation*}

%

\end{proof}

\subsection{Compactness on weighted spaces}

We can now prove the compactness of Calder\'on-Zygmund operators on weighted spaces.

\begin{proof}[Proof of Theorem \ref{CZOWeightedCompactness}]
Let $\tilde T=T-\Pi^{*}_{P_N(T^*1)}$. Since $\Pi^{*}_{P_N(T^*1)}$ is of finite rank, showing that $T$ is compact on $L^2(w)$ is equivalent to showing that $\tilde T$ is compact on $L^2(w)$.
In particular, we argue that for each $\varepsilon>0$, there exists $N_0>0$ such that $$\|P_N^{\perp}\tilde Tf\|_{L^p(w)} \lesssim \varepsilon 
[w]_{A_p}^{\max\left\{1,\frac{p'}{p}\right\}}\|f\|_{L^p(w)}$$ 
for all $N>N_0$ and all $f\in L^p(w)$.

We provide a sketch of the proof using the reasoning of Theorem \ref{Haarboundedness}. By Theorem \ref{domtilde}, there exist $N_0>0$ and a sparse family of cubes ${\mathcal S}$ such that  
$$
	|P_N^\perp \tilde Tf(x)|\lesssim \varepsilon \sum_{R \in \mathcal{S}}  \langle |f| \rangle_{R^*} \mathbbm{1}_R(x) =: \varepsilon S|f|(x)
$$
for all $N>N_0$ and almost every $x \in \mathbb{R}^n$.

Let first $p\ge 2$ and set $\sigma=w^{1-p'}$. We use again  
$
    \|T\|_{L^p(w)\rightarrow L^p(w)} = \|T(\cdot\,\sigma)\|_{L^p(\sigma)\rightarrow L^p(w)}
$
and proceed by duality. Let $f\in L^p(\sigma)$ and $g\in L^{p'}(w)$ be nonnegative. 

For each $R\in {\mathcal S}$ we denote by $E(R)$ the set described in 
Theorem \ref{Haarboundedness} that satisfies $E(R)\subseteq R$, $|R|\leq 2 |E(R)|$, and such that given 
$R,R'\in {\mathcal S}$ with $R\neq R'$, the corresponding sets $E(R)$ and $E(R')$ are disjoint. We use these properties, the $A_p$ condition for $w$, the containment $R\subsetneq R^*$, the inequality $|R^*|\lesssim |R|$, and boundedness of the maximal functions $M_{\sigma}$ and $M_{w}$ from Lemma \ref{WeightedMaxBoundedness}, to obtain
\begin{align*}
    \varepsilon \langle S(f\sigma),gw\rangle
    &= \varepsilon \sum_{R\in \mathcal{S}} \langle f\sigma\rangle_{R^*}\int_{R} gw\,dm \\
    &\leq \varepsilon\sum_{R\in \mathcal{S}} \frac{w(R^*)\sigma(R^*)^{p-1}}{|R^*|^p}\frac{|R^*|^{p-1}}{w(R)\sigma(R^*)^{p-1}}\int_{R^*}f\, d\sigma  \int_{R}g\,dw\\
    &\lesssim \varepsilon [w]_{A_p} \sum_{R\in \mathcal{S}}\left(\frac{1}{\sigma(R^*)}\int_{R^*}f\, d\sigma \right)\left(\frac{1}{w(R)}\int_{R}g\,dw\right)|R|^{p-1}\sigma(R^*)^{2-p}\\
    &\leq \varepsilon 2^{p-1}[ w]_{A_p}\sum_{R\in \mathcal{S}}
\langle f\rangle_{R^*, d\sigma } \langle g \rangle_{R, dw} 
    |E(R)|^{p-1}\sigma(E(R))^{2-p}\\
    &\lesssim \varepsilon [w]_{A_p}\sum_{R\in \mathcal{S}}\langle f\rangle_{R^*, d\sigma} \langle g \rangle_{R, dw}
    w(E(R))^{\frac{1}{p'}}\sigma(E(R))^{\frac{1}{p}}\\
    &\leq \varepsilon [w]_{A_p}\left(\sum_{R\in \mathcal{S}}
    \langle f\rangle_{R^*, d\sigma }^p
    \sigma(E(R))\right)^{\frac{1}{p}}\left(\sum_{R}\langle g \rangle_{R, dw}^{p'}w(E(R))\right)^{\frac{1}{p'}}\\
    &
    \lesssim \varepsilon [w]_{A_p}\|M_{\sigma}f\|_{L^p(\sigma)}\|M_wg\|_{L^{p'}(w)}\\
    &\lesssim \varepsilon [w]_{ A_p}\|f\|_{L^p(\sigma)}\|g\|_{L^{p'}(w)}.
\end{align*}
The case $1<p<2$ follows from duality exactly as in the proof of Theorem \ref{Haarboundedness}.



\end{proof}


\begin{bibdiv}
\begin{biblist}
\bib{CAP2019}{article}{
title={Nondoubling Calder\'on-Zygmund theory: a dyadic approach},
author={J. M. Conde-Alonso},
author={J. Parcet},
journal={J. Fourier Anal. Appl.},
volume={25},
date={2019},
number={4},
pages={1267--1292}
}

\bib{CAR2016}{article}{
title={A pointwise estimate for positive dyadic shifts and some applications},
author={J. M. Conde-Alonso},
author={G. Rey},
journal={Math. Ann.},
volume={365},
date={2016},
number={3-4},
pages={1111--1135}
}


\bib{HMW1973}{article}{
title={Weighted norm inequalities for the conjugate function and Hilbert transform},
author={R. Hunt},
author={B. Muckenhoupt},
author={R. Wheeden},
journal={Trans. Amer. Math. Soc.},
volume={176},
date={1973},
pages={227--251}
}

\bib{HL2020}{article}{
title={Extrapolation of compactness on weighted spaces},
author={T. P. Hyt\"onen},
author={S. Lappas},
date={2020},
journal={Arxiv e-prints: 2003.01606}
}

\bib{H2012}{article}{
title={The sharp weighted bound for general Calder\'on-Zygmund operators},
author={T. P. Hyt\"onen},
journal={Ann. of Math. (2)},
volume={175},
date={2012},
number={3},
pages={1473--1506}
}

\bib{HRT2017}{article}{
title={Quantitative weighted estimates for rough homogeneous singular integrals},
author={T. Hyt\"onen},
author={L. Roncal},
author={O. Tapiola},
journal={Israel J. Math.},
volume={218},
number={1},
date={2017},
pages={133--164}
}


\bib{L2017}{article}{
title={An elementary proof of the $A_2$ bound},
author={M. T. Lacey},
journal={Israel J. Math.},
volume={217},
date={2017},
pages={181--195}
}


\bib{L2013}{article}{
title={A simple proof of the $A_2$ conjecture},
author={A. K. Lerner},
journal={Int. Math. Res. Not.},
volume={\,},
date={2013},
number={14},
pages={3159--3170}
}

\bib{LN2019}{article}{
title={Intuitive dyadic calculus: the basics},
author={A. K. Lerner},
author={F. Nazarov},
journal={Expo. Math.},
volume={37},
date={2019},
number={3},
pages={225--265}
}

\bib{LO2019}{article}{
title={Some remarks on the pointwise sparse domination},
author={A. K. Lerner},
author={S. Ombrosi},
journal={J. Geom. Anal.},
date={2019},
pages={1--17}
}


\bib{LSMP2014}{article}{
title={Dyadic harmonic analysis beyond doubling measures},
author={L. D. L\'opez-S\'anchez},
author={J. M. Martell},
author={J. Parcet},
journal={Adv. Math.},
volume={267},
date={2014},
pages={44--93}
}

\bib{M2012}{article}{
title={Sharp weighted bounds without testing or extrapolation},
author={K. Moen},
journal={Arch. Math. (Basel)},
volume={99},
date={2012},
number={5},
pages={457--466}
}


\bib{OV2017}{article}{
title={Endpoint estimates for compact Calder\'on-Zygmund operators},
author={J-F. Olsen},
author={P. Villarroya},
journal={Rev. Mat. Iberoam.},
volume={33},
date={2017},
pages={1285–-1308}
}

\bib{PPV2017}{article}{
title={Endpoint compactness of singular integrals and perturbations of the Cauchy integral},
author={K-M. Perfekt},
author={S. Pott},
author={P. Villarroya},
journal={Kyoto J. Math.},
volume={57},
date={2017},
number={2},
pages={365--393}
}


\bib{TTV2015}{article}{
title={Weighted martingale multipliers in the non-homogeneous setting and outer measure spaces},
author={C. Thiele},
author={S. Treil},
author={A. Volberg},
journal={Adv. Math.},
volume={285},
date={2015},
pages={1155--1188}
}

\bib{V2015}{article}{
title={A characterization of compactness for singular integrals},
author={P. Villarroya},
journal={J. Math. Pures Appl.},
volume={104},
date={2015},
pages={485--532}
}

\bib{V2019}{article}{
title={A global $Tb$ theorem for compactness and boundedness of Calder\'on-Zygmund operators},
author={P. Villarroya},
journal={J. Math. Anal. Appl.},
volume={480},
date={2019},
number={1}
}

\end{biblist}
\end{bibdiv}
\end{document}